\documentclass[a4paper]{article}

\usepackage{times,cite}
\usepackage{amsfonts,amsthm,amssymb,amsmath}
\usepackage[shortlabels]{enumitem}
\usepackage{multirow}
\usepackage{placeins}
\usepackage{stmaryrd}
\usepackage[T1]{fontenc}

\begin{document}

\theoremstyle{definition}
\newtheorem{defs}{Definition}[section]

\theoremstyle{plain}
\newtheorem{prop}[defs]{Proposition}
\newtheorem{thm}[defs]{Theorem}
\newtheorem{cor}[defs]{Corollary}
\newtheorem{lem}[defs]{Lemma}

\newcommand{\der}[1]{Definition~\ref{#1}}
\newcommand{\propr}[1]{Proposition~\ref{#1}}
\newcommand{\thmr}[1]{Theorem~\ref{#1}}
\newcommand{\corr}[1]{Corollary~\ref{#1}}
\newcommand{\lemr}[1]{Lemma~\ref{#1}}
\newcommand{\secr}[1]{Section~\ref{#1}}

\title{Computing the density of tautologies in propositional logic by solving system of quadratic equations of generating functions}

\author{Taehyun Eom}

\maketitle

\begin{abstract}
    In this paper, we will provide a method to compute the density of tautologies among the set of well-formed formulae consisting of $m$ variables, the negation symbol and the implication symbol; which has a possibility to be applied for other logical systems. This paper contains computational numerical values of the density of tautologies for two, three, and four variable cases. Also, for certain quadratic systems, we will build a theory of the $s$-cut concept to make a memory-time trade-off when we compute the ratio by brute-force counting, and discover a fundamental relation between generating functions' values on the singularity point and ratios of coefficients, which can be understood as another interpretation of the Szeg\H{o} lemma for such quadratic systems. With this relation, we will provide an asymptotic lower bound $m^{-1}-(7/4)m^{-3/2}+O(m^{-2})$ of the density of tautologies in the logic system with $m$ variables, the negation, and the implication, as $m$ goes to the infinity.
\end{abstract}

\section{Introduction}\label{intro}
Propositional logic is one of the most basic theories of mathematics, and tautology is one of the most important concepts of it. In certain propositional logic systems, the set of tautologies and the set of theorems are equal, so understanding proofs is strongly related to understanding tautologies. For example, in the Hilbert deduction system with logical symbols $\neg$ (negation) and $\to$ (implication); an inference rule modus ponens, which means $\phi\to\psi$ and $\phi$ proves $\psi$; if we take axiom schemes as
\begin{itemize}
\item $\phi\to[\psi\to\phi]$,
\item $[\phi\to[\psi\to\xi]]\to[[\phi\to\psi]\to[\phi\to\xi]]$,
\item $[\neg\psi\to\neg\phi]\to[\phi\to\psi]$,
\end{itemize}
the given propositional logic system is well-known to be sound, which means every theorem is a tautology, and complete, which means every tautology is a theorem. This kind of complete propositional logic systems are widely accepted as a logical base for modern mathematics. Here, first two axiom schemes are equivalent to the deduction theorem: under fixed assumptions, $\phi$ proves $\psi$ if and only if $\phi\to\psi$ is provable. Moreover, by the Curry-Howard correspondence, $\phi\to[\psi\to\phi]$ corresponds to the $K$ combinator $Kxy=x$ and $[\phi\to[\psi\to\xi]]\to[[\phi\to\psi]\to[\phi\to\xi]]$ corresponds to the $S$ combinator $Sxyz=xz(yz)$ where those $S$, $K$ combinators are considered as one of the simplest Turing complete language.

So we have some possibility that tautologies and theorems are equivalent, but indeed, their structures are quite different. For example, the truth value of a well-formed formula in a propositional logic system is completely determined by its subformulae and truth tables of logical operators, so a decision algorithm for tautologies is simply constructible and being a tautolgy is universal for most logic systems where a given formula is valid. But provability cannot be determined by its subformulae, so it is hard to design a decision algorithm for theorems and being a theorem completely depends on which logic system is considered for. Hence, if we want to generate randomized proofs as a first step of AI-based proof, it is helpful to understand the proportion of theorems clearly, and we can replace it to the proportion of tautologies, which is simpler, in such complete propositional logic systems.

Also, even though the structure of tautologies is simpler, we still have several interesting questions about tautology itself, not only related to the theorem. 3-SAT problem is one of the example that can be understood as a tautology decision problem, since satisfiability is equivalent to non-tautology-ness of its negation. For example, $(x\vee y\vee z)\wedge (x\vee y\vee \neg z)$ is satisfiable if and only if $(\neg x\wedge\neg y\wedge\neg z)\vee (\neg x\wedge\neg y\wedge z)$ is not a tautology. Hence, understanding tautologies is strongly related to understanding the satisfiability.

Therefore, there exist some preceding studies such as \cite{Za}, which computes the density of tautologies in the logic system with implication and negation on one variable, \cite{Fo}, which computes the asymptotic density, as the number of variables goes to infinity, of tautologies in the logic system with implication and negative literals, and \cite{As}, which computes the densities in several logic systems based on one variable, etc. Here, the density means the limit probability of tautologies among fixed length well-formed formulae where the length goes to the infinity.

Among those studies, many of them such as \cite{Fo, Ch, Fo2, Fo3, Ge, Ge2, Ge3, Ko, Mo, De}, focus on logic systems without $\neg$, where this negation symbol makes structure complicate since it is a unary operator, so it changes Catalan-like structure to Motzkin-like structure like in \cite{Bo, Gr}. \cite{Za} and \cite{As} are studies cover the $\neg$ symbol, but they only focus on the one variable case.

Hence, this paper will cover propositional logic systems with the negation and multivariables. Results mainly consider the system with the negation and the implication to focus on the Hilbert deduction system, but methods can be easily applicable to other logic systems such as one contains the `and' or the `or', etc. \secr{basicdefs} is for basic definitions about propositional logic systems. This section covers the reason why the density exists, and how to construct the system of quadratic equations of generating functions properly to represent tautologies.

\secr{exactcomp} will provide a divide and conquer method to solve the given system of equations exactly by introducing well-organized partitions, and an algorithm computes the density of tautologies in the logic system with $m$ variable. Well-organized partition is a coset-like object for set operators, and it will give a proper way to cluster the system of equations hierarchically. As a result, we can compute the density of tautologies by solving quadratic equations repeatedly, and hence, the density will be a constructible number.

\secr{analytic} will provide an analytic approach to general systems of quadratic equations to make a memory-time trade-off to estimate the limit ratio of coefficients. We will construct and develop a theory of the `$s$-cut' concept which gives a heuristic method to compute an approximation. There is still some remaining problems to complete the theory rigorously, but we can observe that it gives practically improved results with same amount of informations.

\secr{asymptotic} will analyze the asymptotic behavior of the density of tautologies. For the propositional logic system with the negation, the implication, and $m$ variables, the density of tautologies have asymptotic lower bound $m^{-1}-(7/4)m^{-3/2}+O(m^{-2})$. Also, this section includes some reasonable evidences to conjecture this lower bound $m^{-1}$ is actually tight asymptotic behavior.

\section{Basic Definitions}\label{basicdefs}
A propositional logic system consists of a set of variables $X$ and logical operators, such as $\neg$ (negation), $\to$ (implication), $\wedge$ (conjunction), $\vee$ (disjunction), $|$ (NAND), etc. Here, $\neg$ is a unary operator, and others are binary operators. These operators recursively define well-formed formulae. Also, each operator has its own truth table, and these truth tables recursively extend each truth assignment $v:X\to\{T\textrm{ (true) }, F\textrm{ (false) }\}$ to the truth valuation of well-formed formulae. In other words, the valuation $\llbracket\phi;v\rrbracket$ is well-defined for any well-formed formula $\phi$ and truth assignment $v$. Moreover, we can give values as $T=1$, $F=0$, so algebraic expressions such as $\llbracket\neg\phi;v\rrbracket=1-\llbracket\phi;v\rrbracket$ and $\llbracket\phi\to\psi;v\rrbracket=1 - \llbracket\phi;v\rrbracket(1 - \llbracket\psi;v\rrbracket)$ are allowed for convenience.

Now, let $\mathcal{VA}$ be the set of truth assignments, $\mathcal{W}$ be the set of well-formed formulae. Then, for any $\phi\in\mathcal{W}$, let the falsity set $F_\phi=\{v\in\mathcal{VA}\mid \llbracket\phi;v\rrbracket=F\}$ and for any $v\in\mathcal{VA}$, let the set of false formulae as $F^{v}=\{\phi\in\mathcal{W}\mid \llbracket\phi;v\rrbracket=F\}$. Similarly, we define $T_\phi$ and $T^v$. These $F_\bullet$, $T_\bullet$ convert a logical operator to a set operator. For example, $F_{\phi\to\psi}=F_\psi\setminus F_\phi$.

Here, since the valuation only relies on truth tables, a logical operator is actually well-defined on $\mathcal{P}(\mathcal{VA})$ as a set operator. For example, we may define $A\to B:=B\setminus A$ so we have $F_{\phi\to\psi}=F_\phi\to F_\psi$. It is remarkable that in this definition using $F_\bullet$, $A\vee B = A\cap B$ and $A\wedge B=A\cup B$. Lastly, for a well-formed formula $\phi$, let $X[\phi]$ be the set of variables occur in $\phi$ and $A[\phi]=A\cap X[\phi]$ for any $A\subseteq X$. Then, for two truth assignments $u,v$, if $u|_{X[\phi]}=v|_{X[\phi]}$, it implies $\llbracket\phi;u\rrbracket=\llbracket\phi;v\rrbracket$.

If every variable is free, then there exists a natural bijection between truth assignments and subsets of variables. In particular, a truth assignment $v$ corresponds to the set $\{x\in X\mid v(x)=T\}$. In other words, we may assume $\mathcal{VA}=\mathcal{P}(X)$. Also, even we admit some limited variables such as negative literals, $\bot$ (false connective), or $\top$ (true connective), we still have natural correspondence between possible truth assignments and feasible subsets of variables. For example, suppose $A\subseteq X$ is a feasible subset of variables. If we have the negative literal $\overline{x}$ for a variable $x$, then we have a condition $|A\cap\{x,\overline{x}\}|=1$. For the $\bot$, $\bot\not\in A$ and for the $\top$, $\top\in A$. Hence, $\llbracket\phi;A\rrbracket$ is also valid for appropriate $A\subseteq X$. Especially, for the case that every variable is free, we have $\llbracket\phi;A\rrbracket=\llbracket\phi;A[\phi]\rrbracket$.

We have two important categories of well-formed formulae: tautologies and antilogies. A tautology is a well-formed formula which is true for every valuations, and an antilogy is a well-formed formula which is false for every valuations. In other words, a well-formed formula $\phi$ is a tautology if and only if $F_\phi=\emptyset$, which is equivalent to $\phi\in\cap_{v\in\mathcal{VA}}T^v$, and an antilogy if and only if $F_\phi=\mathcal{VA}$, which is equivalent to $\phi\in\cap_{v\in\mathcal{VA}}F^v$.

Now, for any well-formed formula $\phi$, we define its length $\ell(\phi)$ as usual: the number of all letters used in it except parentheses, which is used only for preventing ambiguity of the infix notation. Also, if every logical operator is a binary operator, then the length $\ell$ is always odd, so it is natural to use the reduced length $\ell_2(\phi)=\frac{\ell(\phi)+1}{2}$ since it gives consecutive values. With the proper choice of a consecutive length from $\ell$ and $\ell_2$, we have the following.

\begin{lem}\label{lem:consecutive}
If a propositional logic system have the implication $\to$ or an equivalent expression with fixed length such as $\neg\phi\vee\psi$, then we have a constant $N$ such that for any $n\geq N$, there exists a tautology of length $n$. 
\end{lem}
\begin{proof}
Note that for any variable $p$ and well-formed formula $\phi$, $p\to[\phi\to p]$ is a tautology. Since $\ell(p\to[\phi\to p])=\ell(\phi)+4$ and $\ell_2(p\to[\phi\to p])=\ell_2(\phi)+2$; and we use the consecutive length, so is done.
\end{proof}

Now, we will count the number of well-formed formulae, and hence, we need to consider only finite variables to get meaningful result. Let $W(z)$ be the generating function of well-formed formulae made by $\ell$, i.e., $W(z)=\sum_{\phi\in\mathcal{W}}z^{\ell(\phi)}$. Then, we have
\begin{align*}
W(z)=|X|z
&+ (\textrm{the number of unary operators})zW(z)\\
&+ (\textrm{the number of binary operators})zW(z)^2. 
\end{align*}
Also, for the case without unary operator and using $\ell_2$, we have
\begin{displaymath}
W_2(z)=|X|z + (\textrm{the number of binary operators})W_2(z)^2. 
\end{displaymath}

Here, this equation for $W_2$ actually does not determine $W_2$ uniquely, so we need an additional condition $W_2(0)=0$. In other words, it is natural to consider $\widetilde{W_2}(z)=\frac{W_2(z)}{z}$ which satisfies
\begin{displaymath}
\widetilde{W_2}(z) = |X| + (\textrm{the number of binary operators})z\widetilde{W_2}(z)^2.
\end{displaymath}

From now on, we will mainly focus on the specific propositional logic system, with $m$ variables $X=\{x_0,x_1,\cdots,x_{m-1}\}$, $\neg$, and $\to$. Since $W(z)=mz+zW(z)+zW(z)^2$, we have
\begin{displaymath}
W(z)=\frac{1-z-\sqrt{(1-(1+2\sqrt{m})z)(1-(1-2\sqrt{m})z)}}{2z}.
\end{displaymath}

Before we compute the density of tautologies, we first consider the following lemma.
\begin{lem}
Suppose $a_n$, $b_n$ are positive sequences, $\sum_{n=0}^\infty b_n=\infty$ and $\lim_{n\to\infty}\frac{a_n}{b_n}=r$. Then,
\begin{displaymath}
\lim_{n\to\infty}\frac{\sum_{k=0}^n a_k}{\sum_{k=0}^n b_k}=r.
\end{displaymath}
\end{lem}
Hence, we have
\begin{align*}
\lim_{n\to\infty}
&\frac{\textbf{Number of tautolgies with length at most }n}{\textbf{Number of well-formed formulae with length at most }n}\\
&=\lim_{n\to\infty}\frac{\textbf{Number of tautolgies with length }n}{\textbf{Number of well-formed formulae with length }n},
\end{align*}
if the later one exists. Thus, we will compute the later one as done in \cite{As} and \cite{Za}, and call it as the density of tautologies.

Now, to compute the density, we will consider the generating function of tautologies. For any $A\subseteq\mathcal{VA}$, let
\begin{displaymath}
W_A(z)=\sum_{\phi\in\mathcal{W}, F_\phi=A}z^{\ell(\phi)}.
\end{displaymath}
Hence, $W_\emptyset(z)$ is the generating function of tautologies. These generating functions form a following system of equations.
\begin{itemize}
\item If $A=F_{x_i}=\mathcal{P}(X\setminus\{x_i\})$ for some $x_i\in X$, then
\begin{displaymath}
W_A(z)= z + zW_{\neg A}(z) + \sum_{B\to C=A}zW_{B}(z)W_{C}(z) = z + zW_{A^c}(z) + \sum_{C\setminus B=A}zW_{B}(z)W_{C}(z)
\end{displaymath}
\item Otherwise,
\begin{displaymath}
W_A(z)= zW_{\neg A}(z) + \sum_{B\to C=A}zW_{B}(z)W_{C}(z) = zW_{A^c}(z) + \sum_{C\setminus B=A}zW_{B}(z)W_{C}(z).
\end{displaymath}
\end{itemize}

For this system of equations, Drmota-Lalley-Woods theorem (\cite{Fl}, pp. 489) ensure the existence of the ratio $\lim_{n\to\infty}\frac{[z^n]W_\emptyset(z)}{[z^n]W(z)}$.
\begin{thm}[Drmota-Lalley-Woods theorem]
If a nonlinear polynomial system, which means $n$ equations for $y_1(z),\cdots,y_n(z)$ of the form $y_i=f_i(y_1,\cdots,y_n)$ where $f_i$'s are polynomials, satisfies the following:
\begin{itemize}
\item The system uniquely determines the formal power series solution $y_1(z)$. $y_2(z)$, $\cdots$, $y_n(z)$ and that solution has no negative coefficients.
\item The dependency graph of the system is strongly connected. Here, the dependency graph is a directed graph among $y_i$'s where arcs represent appearances in equations.
\item For each $y_i$, there exists $N_i$ such that $n\geq N_i$ implies $[z^n]y_i>0$.
\end{itemize}
then, every $y_i$ has a common radius of convergence $\rho<\infty$ and for each $y_i$, there exists a function $h_i$ analytic at the origin and $y_i(z)=h_i(\sqrt{1-\frac{z}{\rho}})$ near $\rho$. Moreover, we have
\begin{displaymath}
[z^n]y_i \simeq \frac{1}{n\sqrt{n}\rho^n}\sum_{k\geq 0}\frac{d_k}{n^k}.
\end{displaymath}
\end{thm}

In our logic system, with the negation and the implication, we may check conditions easily.
\begin{itemize}
\item From our equation, $[z^n]W_A(z)$ values are uniquely determined by $[z^k]W_B(z)$ for every $k<n$ and $B\subseteq\mathcal{VA}$, so each $W_A(z)$ is uniquely determined.
\item For any $A\subseteq\mathcal{VA}$, $\emptyset\to A=A\setminus\emptyset=A$ and $A\to\emptyset=\emptyset\setminus A=\emptyset$, so the dependency graph is strongly connected since we have both arcs $(W_\emptyset, W_A)$ and $(W_A, W_\emptyset)$.
\item At first, for any $A\subseteq\mathcal{VA}$, there exists a well-formed formula $\phi$ with $F_\phi=A$. Now, for any tautology $\tau$, $F_{\tau\to\phi}=A$, also. Hence, by the \lemr{lem:consecutive}, there exists $N_A$ such that $n\geq N_A$ implies there exists a well-formed formula $\psi$ with length $n$ and $F_\psi=A$.
\end{itemize}
Hence, the density of tautologies is well-defined.

Moreover, to compute the exact value of the density, we will introduce the Szeg\H{o} lemma.
\begin{lem}
Suppose that two power series $U(z)$, $V(z)$ have the common radius of convergence $\rho$, this $\rho$ is the unique simple singularity in the disk $|z|\leq|\rho|$ for them, we have series expansion $U(z)=\sum_{k\geq 0}\widehat{u_k}\left(\sqrt{1-\frac{z}{\rho}}\right)^k$, $V(z)=\sum_{k\geq 0}\widehat{v_k}\left(\sqrt{1-\frac{z}{\rho}}\right)^k$ near $\rho$, and $\widehat{v_1}\not=0$. Then, $\lim_{n\to\infty}\frac{[z^n]U(z)}{[z^n]V(z)}=\frac{\widehat{u_1}}{\widehat{v_1}}$. Thus, we have
\begin{displaymath}
\lim_{n\to\infty}\frac{[z^n]U(z)}{[z^n]V(z)}=\lim_{z\to\rho^{-}}\frac{\frac{U(z)-U(\rho)}{\sqrt{1-z/\rho}}}{\frac{V(z)-V(\rho)}{\sqrt{1-z/\rho}}}=\lim_{z\to\rho^{-}}\frac{-2\rho U'(z)\sqrt{1-z/\rho}}{-2\rho V'(z)\sqrt{1-z/\rho}}.
\end{displaymath}
\end{lem}

Since $W(z)=\frac{1-z-\sqrt{(1-(1+2\sqrt{m})z)(1-(1-2\sqrt{m})z)}}{2z}=\sum_{A\subseteq\mathcal{VA}}W_A(z)$, $W$ also has the common radius of convergence. Hence, our target radius of convergence $\rho$ is $\frac{1}{2\sqrt{m}+1}$, and we need to analyze $W_\emptyset(z)$ near $\frac{1}{2\sqrt{m}+1}$ to compute the density of tautologies.

\section{Exact Computation}\label{exactcomp}

Since our system of equations has $2^{|\mathcal{VA}|}=2^{2^m}$ unknowns, we will try to solve it by divide and conquer method.

\begin{defs}
A partition $\{P_1,\cdots, P_n\}$ of $\mathcal{P}(\mathcal{VA})$ is said to be \textbf{well-organized for an operation} if the given operator is well-defined among parts in the partition. For example, the partition is well-organized for $\to$ if for any $1\leq i,j\leq n$, there exists $k$ such that for any $A\in P_i$ and $B\in P_j$, we have $A\to B\in P_k$. In this case, we will denote this as $P_i\to P_j=P_k$.
\end{defs}
This is a reminiscent of a quotient group in group theory. Hence, we will define a reminiscent concept for the coset as the following.

\begin{defs}
For any $A,B\subseteq\mathcal{VA}$, the standard subclass of $\mathcal{P}(\mathcal{VA})$ associated to $(A,B)$ is defined as
\begin{displaymath}
\mathcal{I}_{A;B}=\{Y\mid A\setminus B\subseteq Y\subseteq A\cup B\}.
\end{displaymath}
Here, if $A\cap B=\emptyset$, then $\mathcal{I}_{A;B}=\{A\cup Y\mid Y\subseteq B\}=\{C\subseteq\mathcal{VA}\mid C\setminus B=A\}$, and by the definition, $\mathcal{I}_{A;B}=\mathcal{I}_{A\setminus B;B}=\mathcal{I}_{A\cup B;B}$.

Then, let \textbf{the standard subclass partition associated to $B$} as
\begin{displaymath}
\mathcal{I}_{;B}=\{\mathcal{I}_{A;B}\mid A\cap B=\emptyset\}.
\end{displaymath}
It is easy to check that this partition is naturally induced by the equivalence relation $A\sim_B C\Leftrightarrow A\setminus B=C\setminus B\Leftrightarrow A\cup B=C\cup B\Leftrightarrow (A\setminus C)\cup(C\setminus A)\subseteq B$.
\end{defs}

\begin{prop}\label{propPart}
Standard subclasses satisfy the following equalities.
\begin{enumerate}[(a)]
\item $\{Y^c\mid Y\in \mathcal{I}_{A;B}\}=\mathcal{I}_{A^c;B}=\mathcal{I}_{\neg A;B}$,
\item $\{Y\cap Z\mid Y\in \mathcal{I}_{A;B},Z\in \mathcal{I}_{C;B}\}=\mathcal{I}_{A\cap C;B}=\mathcal{I}_{A\vee C;B}$,
\item $\{Y\cup Z\mid Y\in \mathcal{I}_{A;B},Z\in \mathcal{I}_{C;B}\}=\mathcal{I}_{A\cup C;B}=\mathcal{I}_{A\wedge C;B}$,
\item $\{Y\setminus Z\mid Y\in \mathcal{I}_{A;B},Z\in \mathcal{I}_{C;B}\}=\mathcal{I}_{A\setminus C;B}=\mathcal{I}_{C\to A;B}$,
\item $\mathcal{I}_{A;B}\cup \mathcal{I}_{A\cup\{y\};B} = \mathcal{I}_{A;B\cup\{y\}}$ if $y\not\in A\cup B$.
\end{enumerate}
\end{prop}
\begin{proof}
(a) For every $Y\in \mathcal{I}_{A;B}$, we have $Y^c\setminus B=Y^c\cap B^c=(Y\cup B)^c=(A\cup B)^c=A^c\setminus B$, and so, $Y^c\in\mathcal{I}_{A^c;B}$. Then, the first equality is obtained by $|\mathcal{I}_{A;B}|=|\mathcal{I}_{A^c;B}|$.

(b), (c), (d) can be done similarly.

(e) $\mathcal{I}_{A;B\cup\{y\}}=\{(A\setminus B)\cup Y, (A\setminus B)\cup Y\cup\{y\}\mid Y\subseteq B\}=\mathcal{I}_{A;B}\cup \mathcal{I}_{A\cup\{y\};B}$.
\end{proof}
Hence, $\mathcal{I}_{;B}$ is well-organized for $\neg$, $\to$, $\vee$ and $\wedge$, and it really works as a `quotient poset' isomorphic to $\mathcal{P}(\mathcal{VA}\setminus B)$ preserves set operations. Also, $\mathcal{I}_{A;B}$ is a translation of $\mathcal{P}(B)$, $\mathcal{I}_{\emptyset;B}$ is an order ideal, and $\mathcal{I}_{B^c;B}$ is a filter in the poset $\mathcal{P}(\mathcal{VA})$.

Moreover, we have the following.
\begin{prop} For a partition $P$ of $\mathcal{P}(A)$ where $A$ is a finite set, the followings are equivalent.
\begin{itemize}
\item $P$ is well-organized for $\setminus$.
\item $P$ is well-organized for $\bullet^c$ and $\cup$.
\item $P$ is well-organized for $\bullet^c$ and $\cap$.
\item $P=\mathcal{I}_{;B}$ for some $B\subseteq A$.
\end{itemize}
\end{prop}
\begin{proof}
It is enough to prove that if $P$ is well-organized for $\setminus$, $\bullet^c$, $\cup$ and $\cap$, then $P=\mathcal{I}_{;B}$ for some $B\subseteq A$. At first, since $P$ is a partition, it naturally defines the equivalence relation $\sim$. Now, let $Y\in P$ be the part containing $\emptyset$. Since $\emptyset\cup\emptyset=\emptyset$, $Y$ is closed under union. Hence, if $B=\cup_{Z\in Y}Z$, then $B\in Y$, also. Moreover, for any $C\subseteq B$, we have $C=C\cup\emptyset\sim C\cup B=B$, so $C\in Y$. Hence, $Y=\mathcal{P}(B)$.

Now, for any $C,D\subseteq A$, if $C\setminus B=D\setminus B$, then $C=(C\setminus B)\cup (C\cap B)\sim(D\setminus B)\cup (D\cap B)=D$, so $C\sim D$. Lastly, suppose that $C\sim D$. Then, $C\setminus B\sim C\setminus\emptyset=C\sim D\setminus\emptyset\sim D\setminus B$. Therefore, $\emptyset=C\setminus C\sim(C\setminus B)\setminus(D\setminus B)=(C\setminus D)\setminus B$. Hence, $(C\setminus B)\setminus(D\setminus B)\subseteq B$, which implies $(C\setminus B)\setminus(D\setminus B)=\emptyset$, so we have $C\setminus B\subseteq D\setminus B$. By symmetry, $C\setminus B=D\setminus B$. Thus, $C\sim D$ if and only if $C\setminus B=D\setminus B$, which means $P=\mathcal{I}_{;B}$.
\end{proof}

Now, let $I_{A;B}(z)$ be the generating function
\begin{displaymath}
I_{A;B}(z):=\sum_{Y\in \mathcal{I}_{A;B}} W_Y(z).
\end{displaymath}
By the definition, for any $Y\subseteq\mathcal{VA}$, $W_Y(z)=I_{Y;\emptyset}(z)$ and $W(z)=I_{\emptyset;\mathcal{VA}}$.

For fixed $B\subseteq\mathcal{VA}$, these $I_{A;B}(z)$ form a following system of equations,
\begin{displaymath}
I_{A;B}(z)=\left|\{x\in X\mid F_x\in \mathcal{I}_{A;B}\}\right|z + zI_{A^c;B}(z) + \sum_{\underset{C\cap B=D\cap B=\emptyset}{C\setminus D=A\setminus B}}zI_{C;B}(z)I_{D;B}(z),
\end{displaymath}
which is similar to the system of equations for $W_{A}$'s. Here, the number of unknowns is reduced to $2^{2^m - |B|}$. Hence, we will try to solve these equations from large $B$'s to small $B$'s. Now, we may count the number of pairs $(C,D)$ such that $C\cap B=D\cap B=\emptyset$ and $C\setminus D=A\setminus B$. By considering the map $(C,D)\mapsto (C\cap D, D\setminus C, (C\cup D)^c)$, the number of such pair $(C,D)$ is equal to the number of triplets $(C\cap D, D\setminus C, (C\cup D)^c)$ partitioning $\mathcal{VA}\setminus (A\cup B)$. Hence, the number is $3^{|\mathcal{VA}|-|A\cup B|}=3^{|(A\cup B)^c|}$. Thus, it is worth to define
\begin{displaymath}
I_{-;B}(z):=I_{B^c;B}(z)
\end{displaymath}
for convenience.

\begin{prop}\label{lincomprop}
For any $A,B\subseteq\mathcal{VA}$, we have the following:
\begin{enumerate}[label=(\alph*)]
\item $I_{A;B}(z)$ is a linear combination of elements of the set
\begin{displaymath}
\{I_{\emptyset;B}(z)\}\cup\{I_{C;B'}(z)\mid C\subsetneq A\setminus B, \ C\cap B'=\emptyset, \ |B'|=|B|+1, \ B\subseteq B'\},
\end{displaymath}
where the coefficient of $I_{\emptyset;B}(z)$ is $(-1)^{|A\setminus B|}$.
\item $I_{A;B}(z)$ is a linear combination of elements of the set
\begin{displaymath}
\{I_{-;B}(z)\}\cup\{I_{C;B'}(z)\mid A\setminus B\subseteq C, \ C\cap B'=\emptyset, \ |B'|=|B|+1, \ B\subseteq B'\},
\end{displaymath}
where the coefficient of $I_{-;B}(z)$ is $(-1)^{|B^c\setminus A|}$.
\end{enumerate}
\end{prop}
\begin{proof}
(a) It is enough to prove when $A\cap B=\emptyset$. We will induct on $|A|$. It is trivial for $A=\emptyset$. Suppose it is true for every $A$ such that $|A|=n$ and $A\cap B=\emptyset$. Then, if $|A|=n+1$ and $A\cap B=\emptyset$, choose any $y\in A$. Now, we have
\begin{displaymath}
I_{A;B}(z)=I_{A\setminus\{y\};B\cup\{y\}}(z)-I_{A\setminus\{y\};B}(z),
\end{displaymath}
which proves the proposition.

(b) It follows from
\begin{displaymath}
I_{A;B}(z)=I_{A;B\cup\{y\}}(z)-I_{A\cup\{y\};B}(z)
\end{displaymath}
for $y\not\in A\cup B$.
\end{proof}

\begin{cor}\label{lincom}
For any $A,A',B\subseteq\mathcal{VA}$, $I_{A';B}(z)$ is a linear combination of elements of the set
\begin{displaymath}
\{I_{A;B}(z)\}\cup\{I_{C;B'}(z)\mid |B'|=|B|+1, \ B\subseteq B'\},
\end{displaymath}
where the coefficient of $I_{A;B}(z)$ is $\pm 1$.
\end{cor}
\begin{proof}
It directly follows from the above proposition, since the coefficient of $I_{\emptyset;B}$ is $\pm 1$.
\end{proof}

\begin{cor}
For any $A,B\subseteq\mathcal{VA}$, $I_{A,B}(z)$ is a linear combination of
\begin{enumerate}[label=(\alph*)]
\item
\begin{displaymath}
\{I_{\emptyset;B'}(z)\mid B\subseteq B'\subseteq A\cup B\}.
\end{displaymath}
\item
\begin{displaymath}
\{I_{-;B'}(z)\mid (A\setminus B)\cap B'=\emptyset, \ B\subseteq B'\}=\{I_{-;B'}(z)\mid B\subseteq B'\subseteq A^c\cup B\}.
\end{displaymath}
\end{enumerate}
\end{cor}
\begin{proof}
This also directly follows from \propr{lincomprop}.
\end{proof}

The following proposition gives a way to compute exact coefficients when we write $I_{A;B}(z)$ as a linear combination of $I_{\emptyset;B'}(z)$'s or $I_{-;B'}(z)$'s.
\begin{prop}\label{coeffOfIAB}
For any $A,B\subseteq\mathcal{VA}$,
\begin{enumerate}[label=(\alph*)]
\item
\begin{displaymath}
I_{A;B}(z)=(-1)^{|A\cup B|}\sum_{B\subseteq B'\subseteq A\cup B}(-1)^{|B'|}I_{\emptyset;B'}(z).
\end{displaymath}
\item
\begin{displaymath}
I_{A;B}(z)=(-1)^{|A\setminus B|}\sum_{B\subseteq B'\subseteq A^c\cup B}(-1)^{|B'|}I_{-;B'}(z)
\end{displaymath}
\end{enumerate}
\end{prop}
\begin{proof}
(a) First, $C\in \mathcal{I}_{\emptyset;B'}$ if and only if $C\subseteq B'$. Hence,
\begin{displaymath}
\mathcal{I}_{\emptyset;B'}\cap \mathcal{I}_{\emptyset;B''}=\mathcal{I}_{\emptyset;B'\cap B''}
\end{displaymath}
is satisfied for any $B'$ and $B''$. Now, we have
\begin{displaymath}
\mathcal{I}_{A;B}=\mathcal{I}_{\emptyset;A\cup B}\setminus\left(\bigcup_{y\in A\setminus B}\mathcal{I}_{\emptyset;(A\setminus\{y\})\cup B}\right).
\end{displaymath}
Thus, by the inclusion-exclusion principle, we get
\begin{align*}
I_{A;B}(z)&=\sum_{i=0}^{|A\setminus B|}\sum_{Y\subseteq A\setminus B, |Y|=i}(-1)^i I_{\emptyset;(A\setminus Y)\cup B}(z)\\
&=\sum_{Y\subseteq A\setminus B}(-1)^{|Y|}I_{\emptyset;(A\setminus Y)\cup B}(z)\\
&=\sum_{B\subseteq B'\subseteq A\cup B}(-1)^{|A\cup B|-|B'|}I_{\emptyset;B'}(z)\\
&=(-1)^{|A\cup B|}\sum_{B\subseteq B'\subseteq A\cup B}(-1)^{|B'|}I_{\emptyset;B'}(z).
\end{align*}

(b) Similarly, $C\in \mathcal{I}_{-;B'}$ if and only if ${B'}^c\subseteq C$ which is equivalent to $C^c\subseteq B'$. Hence,
\begin{displaymath}
\mathcal{I}_{-;B'}\cap \mathcal{I}_{-;B''}=\mathcal{I}_{-;B'\cap B''}
\end{displaymath}
is satisfied for any $B'$ and $B''$. Now, we have
\begin{displaymath}
\mathcal{I}_{A;B}=\mathcal{I}_{-;A^c\cup B}\setminus\left(\bigcup_{y\in(A\cup B)^c}\mathcal{I}_{-;(A\cup\{y\})^c \cup B}\right).
\end{displaymath}
Thus, by the inclusion-exclusion principle, we get
\begin{align*}
I_{A;B}(z)&=\sum_{i=0}^{|\mathcal{VA}|-|A\cup B|}\sum_{Y\subseteq (A\cup B)^c, |Y|=i}(-1)^i I_{-;(A\cup Y)^c\cup B}(z)\\
&=\sum_{Y\subseteq (A\cup B)^c}(-1)^{|Y|}I_{-;(A\cup Y)^c\cup B}(z)\\
&=\sum_{B\subseteq B'\subseteq A^c\cup B}(-1)^{|{B'}^{c}\setminus (A\setminus B)|}I_{-;B'}(z)\\
&=(-1)^{|A\setminus B|}\sum_{B\subseteq B'\subseteq A^c\cup B}(-1)^{|B'|}I_{-;B'}(z).
\end{align*}
\end{proof}
Note that these results directly come from the definition of $I_{A;B}$, so it does not depend on the logic system.

\begin{thm}
For any $A,B\subseteq\mathcal{VA}$, $I_{A;B}(z)$ is obtained by arithmetic operations and taking square roots. In particular, so is $W_A(z)$.
\end{thm}
\begin{proof}
We will induct on $|B|$ in reverse direction, from the largest to the smallest. If $|B|=|\mathcal{VA}|$ so $B=\mathcal{VA}$, then $I_{\emptyset;B}(z)=W(z)$ which is already known as a composition of arithmetic operations and taking square roots.

Now, assume that it holds for every $B$ with $|B|=n+1$. Then, for the case $|B|=n$, we have a system of equation
\begin{displaymath}
I_{A;B}(z)=\left|\{x\in X\mid F_x\in \mathcal{I}_{A;B}\}\right|z + zI_{A^c;B}(z) + \sum_{\underset{C\cap B=D\cap B=\emptyset}{C\setminus D=A\setminus B}}zI_{C;B}(z)I_{D;B}(z).
\end{displaymath}
Now, by \corr{lincom}, $I_{A^c;B}(z)$, $I_{C;B}(z)$, $I_{D;B}(z)$ are linear combinations of $I_{A;B}(z)$ and $I_{\bullet;B'}(z)$'s where $|B'|=n+1$. Thus, given equation is an at-most-quadratic equation for $I_{A;B}(z)$, and it is nontrivial since the coefficient of $I_{A;B}(z)$ is 1 modulo $z$. Thus, $I_{A;B}(z)$ is again, a composition of arithmetic operations and taking square roots. 
\end{proof}
Indeed, this theorem is also true for other propositional logic systems with at most 2-ary operators. Moreover, note that general systems of quadratic equations are not even solvable by radicals. For example, a system of equation
\begin{displaymath}
\left\{\begin{array}{l}
x = z^2\\
y = x^2\\
z = yz + \rho
\end{array}\right.
\end{displaymath}
is equivalent to the Bring-Jerrard quintic eqaution $z^5-z+\rho=0$, where it is well-known that there is no general solution using radicals for them.

Now, we will compute the exact value of the density of tautologies. At first, we will start from the following equation:
\begin{displaymath}
I_{-;B}=\left|\{x\mid F_{x}\in \mathcal{I}_{-;B}\}\right|z + zI_{\emptyset;B} + zI_{\emptyset;B}I_{-;B},
\end{displaymath}
where we have
\begin{displaymath}
I_{\emptyset;B}=\sum_{B\subseteq B'}(-1)^{|B'|}I_{-;B'}=(-1)^{|B|}I_{-;B}+\sum_{B\subsetneq B'}(-1)^{|B'|}I_{-;B'}.
\end{displaymath}

We will introduce the following definitions:
\begin{itemize}
\item $m_{-;B} := \left|\{x\mid F_{x}\in \mathcal{I}_{-;B}\}\right|$,
\item $\sigma_B := (-1)^{|B|}$,
\item $I_B^{\uparrow}(z):= \sum_{B\subsetneq B'}\sigma_{B'}I_{-;B'}$,
\item $\rho_0 := \frac{1}{2\sqrt{m}+1}$,
\item $\alpha_B := I_{-;B}(\rho_0)$,
\item $\alpha_B^{\uparrow} := I_B^{\uparrow}(\rho_0) = \sum_{B\subsetneq B'}\sigma_{B'}\alpha_{B'}$,
\item $\beta_B := 2\rho_0\lim_{z\to\rho_0^{-}}I_{-;B}'(z)\sqrt{1-\frac{z}{\rho_0}}$,
\item $\beta_B^{\uparrow} := \sum_{B\subsetneq B'}\sigma_{B'}\beta_{B'}$,
\item $D_B(z) := (1-(\sigma_B+I_B^{\uparrow})z)^2-4\sigma_Bz^2(m_{-;B}+I_B^{\uparrow})$.
\item $d_B := \frac{D_B(\rho_0)}{\rho_0^2} = (\frac{1}{\rho_0} - \sigma_B - \alpha_B^{\uparrow})^2 - 4\sigma_B(m_{-;B} + \alpha_B^{\uparrow})$.
\end{itemize}
Since $W(z)=I_{\emptyset;\mathcal{VA}}(z)$, by Szeg\H{o} lemma, we have
\begin{displaymath}
\lim_{n\to\infty}\frac{[z^n]I_{A;B}(z)}{[z^n]W(z)}=\frac{\sigma_{A}\sum_{B\subseteq B'\subseteq A^c}\sigma_{B'}\beta_{B'}}{\beta_{\mathcal{VA}}}
\end{displaymath}
for any disjoint $A,B\subseteq\mathcal{VA}$.

At first, the quadratic equation for $I_{-;B}$ can be simplified as
\begin{displaymath}
I_{-;B}=z(m_{-;B}+I_B^\uparrow) + z(\sigma_B + I_B^\uparrow)I_{-;B} + z\sigma_BI_{-;B}^2.
\end{displaymath}
Hence, we have
\begin{displaymath}
I_{-;B}=\frac{1-(\sigma_B+I_B^\uparrow)z-\sqrt{D_B(z)}}{(2\sigma_B)\,z},
\end{displaymath}
since $I_{-;B}(0)=0$. Then,
\begin{displaymath}
I_{-;B}'(z)=\frac{-(\sigma_B+I_B^\uparrow(z))-z{I_B^\uparrow}'(z)-\frac{D_B'(z)}{2\sqrt{D_B(z)}}}{(2\sigma_B)\,z}-\frac{I_{-;B}(z)}{z},
\end{displaymath}
where
\begin{displaymath}
D_B'(z)=-2(\sigma_B+I_B^\uparrow+z{I_B^\uparrow}')(1-(\sigma_B+I_B^\uparrow)z)-4\sigma_B z^2(m_{-;B}+{I_B^\uparrow}')-8\sigma_Bz(m_{-;B}+I_B^\uparrow).
\end{displaymath}

With these computations, we get the following:
\begin{align*}
\alpha_B &=\frac{2\sqrt{m}+1-\sigma_B-\alpha_B^\uparrow-\sqrt{d_B}}{2\sigma_B},\\
\beta_B &=\beta_B^{\uparrow}\frac{-1+\frac{2\sqrt{m}+1+\sigma_B-\alpha_B^{\uparrow}}{\sqrt{d_B}}}{2\sigma_B},
\end{align*}
where $\alpha_{\mathcal{VA}}=\sqrt{m}$ and $\beta_{\mathcal{VA}}=\sqrt{2m+\sqrt{m}}$. Here, $d_B=0$ only occurs when $\beta_B^\uparrow=0$, which is nothing but $B=\mathcal{VA}$. Here, for each $B\subseteq\mathcal{VA}$, we may use the binary representation to make a correspondence to an integer $b\in[0,2^{2^m})$. With, this representation, $B\subseteq B'$ implies $b\leq b'$, so we can run this algorithm by simple for-loops.

Finally, with these results, we can compute the density of tautologies
\begin{displaymath}
\lim_{n\to\infty}\frac{[z^n]I_{\emptyset;\emptyset}(z)}{[z^n]W(z)}=\frac{\sum_{B\subseteq\mathcal{VA}}\sigma_{B}\beta_{B}}{\sqrt{2m+\sqrt{m}}}
\end{displaymath}
algorithmically. As a corollary, the density is a constructible number. Following table gives its value for $m=1,2,3,4$.
\FloatBarrier
\begin{table}[h!]
\caption{The density of tautologies}
\begin{center}
\begin{tabular}{|c|c|}
\hline
$m=1$ & 0.42324.. \\
$m=2$ & 0.33213.. \\
$m=3$ & 0.27003.. \\
$m=4$ & 0.22561.. \\
\hline
\end{tabular}
\end{center}
\end{table}
\FloatBarrier

\section{Analytic Approach}\label{analytic}

From a quadratic equation for a generating function, to compute the $n$th degree coefficient, we need every informations from the constant term to the $(n-1)$th degree term. On the other hand, from the Drmota-Lalley-Woods theorem which gives
\begin{displaymath}
[z^n]y_i \simeq \frac{1}{n\sqrt{n}\rho^n}\sum_{k\geq 0}\frac{d_k}{n^k},
\end{displaymath}
$\frac{[z^n]I_{\emptyset;\emptyset}(z)}{[z^n]W(z)}$ converges to its limit in $O(\frac{1}{n})$ order, as shown in the following table.
\FloatBarrier
\begin{table}[h!]
\caption{Approximation to the density}
\begin{center}
\begin{tabular}{|c|c|c|c|c|}
\hline
 & true & $n=10$ & $n=50$ & $n=200$\\
\hline
$m=1$ & 0.4232 & 0.3102 & 0.4142 & 0.4210 \\
$m=2$ & 0.3321 & 0.2374 & 0.3206 & 0.3293 \\
$m=3$ & 0.2700 & 0.1913 & 0.2581 & 0.2670 \\
\hline
\end{tabular}
\end{center}
\end{table}
\FloatBarrier
Hence, if there is a way to compute more accurate approximate value with same information, it can work as a memory-time trade-off. This kind of convergence speed problem occurs for any general system of quadratic equations, not only for the system of quadratic equations for tautologies. Thus, we will build a theory which is generally applicable to systems of quadratic equations.

For a polynomial $h$, a power series $Y(z)=\sum_{n=0}^\infty Y_nz^n$ is \textbf{$h$-organized} if
\begin{itemize}
\item There exists a limit ratio $\lim_{n\to\infty}\frac{Y_{n+1}}{Y_n}=\frac{1}{\rho}>1$,
\item There exist a power series $f$ and a polynomial $g$ such that $Y(z)=f(z)+g(z)Y(z)+h(z)Y(z)^2$,
\item There exists a limit ratio $\lim_{n\to\infty}\frac{[z^n]f(z)}{[z^n]Y(z)}=\gamma$.
\end{itemize}
Here, we assume $g$, $h$ as polynomials, so they do not have their own independent combinatorial structures and only represent recursive relations in $Y$.

Now, suppose that power series $A_i(z)=\sum_{n=0}^\infty A_{i\,n}z^n$ for $i=1,\cdots, N$ satisfy the following system of at-most-quadratic relations
\begin{displaymath}
A_i(z)=f_i(z)+\sum_{j=1}^Ng_{ij}(z)A_j(z)+\sum_{j,k=1}^Nh_{ijk}(z)A_j(z)A_k(z),
\end{displaymath}
where each $f_i$ is a power series, and each $g_{ij}$, $h_{ijk}$ is a polynomial. This system of equations is \textbf{$(Y,h)$-organized} if
\begin{itemize}
\item There exist limit ratios $\lim_{n\to\infty}\frac{[z^n]f_i(z)}{[z^n]Y(z)}=\gamma_i$,
\item There exist limit ratios $\lim_{n\to\infty}\frac{[z^n]A_i(z)}{[z^n]Y(z)}=\beta_i$,
\item Each $h_{ijk}(z)$ is divisible by $h(z)$.
\end{itemize}
Note that for such system, we may add $Y(z)$ as $A_0(z)$ to the system freely. Moreover, we may deal with cubic terms $h_{ijkl}A_j A_k A_l$ if $h^2|h_{ijkl}$, by introducing $B_{jk}(z)=h(z)A_j(z)A_k(z)$.

Lastly, for any power series $F(z)=\sum_{n=0}^\infty F_nz^n$, define the \textbf{$s$-cut of $F$} as
\begin{displaymath}
F^{\leq s}(z)=\sum_{n=0}^s F_nz^n.
\end{displaymath}
This $s$-cut of a power series represents computed results for coefficients of $F$. Hence, for any value $r$, $F(r)$ is computable by using known informations saved in the memory.

With these definitions, we will try to compute an approximate value for $\beta_i$'s by $A_i^{\leq s}(z)$, $Y^{\leq s}(z)$. At first, we have
\begin{displaymath}
1=\frac{f_n}{Y_n}+\sum_{u=0}^{\deg g}g_u\frac{Y_{n-u}}{Y_n}+\sum_{u=0}^{\deg h}\sum_{v=0}^{n-u}h_u\frac{Y_v Y_{n-u-v}}{Y_n}.
\end{displaymath}
From this, we get
\begin{displaymath}
\sum_{u=0}^{\deg h}\sum_{v=s+1}^{n-u-s-1}h_u\frac{Y_v Y_{n-u-v}}{Y_n}=1-\frac{f_n}{Y_n}-\sum_{u=0}^{\deg g}g_u\frac{Y_{n-u}}{Y_n}-2\sum_{u=0}^{\deg h}\sum_{v=0}^{s}h_uY_v\frac{Y_{n-u-v}}{Y_n}.
\end{displaymath}
Here, as $n\to\infty$, $\frac{Y_{n-u}}{Y_n}\to\rho^u$, so we expect that
\begin{displaymath}
\sum_{u=0}^{\deg h}\sum_{v=s+1}^{n-u-s-1}h_u\frac{Y_v Y_{n-u-v}}{Y_n}\simeq 1-\gamma-g(\rho)-2h(\rho)Y^{\leq s}(\rho).
\end{displaymath}
Hence, we may define
\begin{displaymath}
\zeta_s := 1-\gamma-g(\rho)-2h(\rho)Y^{\leq s}(\rho)
\end{displaymath}
and
\begin{displaymath}
\zeta_\infty :=\lim_{s\to\infty}\zeta_s = 1-\gamma-g(\rho)-2h(\rho)Y(\rho).
\end{displaymath}
We will call this $\zeta_\infty$ as \textbf{the impurity of the equation $Y(z)=f(z)+g(z)Y(z)+h(z)Y(z)^2$}.

Then, for a $(Y,h)$-organized system of equations for $A_1,\cdots,A_N$, its \textbf{$s$-cut operator $C_s$} is the map \newline $(x_1,\cdots,x_N)\mapsto(c_1,\cdots,c_N)$ defined as
\begin{align*}
c_i = & \gamma_i +\sum_{j=1}^Ng_{ij}(\rho)x_j
+ \sum_{j,k=1}^Nh_{ijk}(\rho)\left(A_j^{\leq s}(\rho)x_k+A_k^{\leq s}(\rho)x_j\right) 
+ \sum_{j,k=1}^N\zeta_s\frac{h_{ijk}}{h}(\rho)x_jx_k.
\end{align*}
and $(\beta_{1}^{(s)},\cdots,\beta_{N}^{(s)})$ is an \textbf{$s$-cut solution} of the given system of equations if it is a fixed point of $C_s$.

\begin{thm}\label{convergence}
Let $h$ be a polynomial without negative coefficients, $Y$ be an $h$-organized power series without negative coefficients, and $Y(\rho)$ is bounded. Now, if $A_1,\cdots,A_N$ form a $(Y,h)$-organized system of equations, then, for the $s$-cut operator $C_s$, we have
\begin{displaymath}
(\beta_1,\cdots,\beta_N)=\lim_{s\to\infty}C_s(\beta_1,\cdots,\beta_N)
\end{displaymath}
\end{thm}
\begin{proof}
At first, we have
\begin{displaymath}
\lim_{n\to\infty}\sum_{u=0}^{\deg h}h_u\sum_{v=s+1}^{n-u-s-1}\frac{Y_vY_{n-u-v}}{Y_n} = 1-\gamma-g(\rho)-2h(\rho)Y^{\leq s}(\rho)=\zeta_s.
\end{displaymath}
Here, $|\zeta_s|\leq 1+\gamma+|g|(\rho)+2h(\rho)Y(\rho)<\infty$.

Now, let $C_s(\beta_1,\cdots,\beta_N)=(c_{1}^{(s)},\dots,c_{N}^{(s)})$. Then, from
\begin{displaymath}
\frac{A_{in}}{Y_n}=\frac{f_{in}}{Y_n}+\sum_{j=1}^n\sum_{u=0}^{\deg g_{ij}}g_{iju}\frac{[z^{n-u}]A_j}{Y_{n-u}}\frac{Y_{n-u}}{Y_n} + \sum_{j,k=1}^N\sum_{u=0}^{\deg h_{ijk}}h_{ijku}\sum_{v=0}^{n-u}\frac{[z^v]A_j\cdot [z^{n-u-v}]A_k}{Y_n}
\end{displaymath}
and
\begin{displaymath}
c_i^{(s)} = \gamma_i +\sum_{j=1}^Ng_{ij}(\rho)\beta_j + \sum_{j,k=1}^Nh_{ijk}(\rho)\left(A_j^{\leq s}(\rho)\beta_k+A_k^{\leq s}(\rho)\beta_j\right) + \sum_{j,k=1}^N\zeta_s\frac{h_{ijk}}{h}(\rho)\beta_j\beta_k,
\end{displaymath}
we get
\begin{displaymath}
\beta_i-c_{i}^{(s)}=\lim_{n\to\infty}\sum_{j,k=1}^N\sum_{t=0}^{\deg h_{ijk}/h}[z^t]\frac{h_{ijk}}{h}\cdot\sum_{u=0}^{\deg h}h_u\sum_{v=s+1}^{n-t-u-s-1}\Delta_{njktuv}
\end{displaymath}
where
\begin{displaymath}
\Delta_{njktuv}=\frac{[z^v]A_{j}\cdot[z^{n-t-u-v}]A_{k}}{Y_n}-\frac{\beta_j\beta_k Y_vY_{n-t-u-v}}{Y_n}.
\end{displaymath}
Now, for any $\epsilon>0$, choose $s$ so that $s\leq\min\{u,v\}$ implies $|\beta_j\beta_k-\frac{A_{jv}A_{ku}}{Y_uY_v}|<\epsilon$ for any $j,k$; and choose $n$ so that $n>2s+\deg h_{ijk} + \deg h$ for every $i,j,k$. With these choices, for any $v$ such that $s<v<n-t-u-s$,
\begin{align*}
|\Delta_{njktuv}|&=\left|\frac{A_{jv}A_{k,n-t-u-v}}{Y_n}-\frac{\beta_j\beta_k Y_vY_{n-t-u-v}}{Y_n}\right|\\
&\leq\epsilon\times\frac{Y_vY_{n-t-u-v}}{Y_{n-t}}\frac{Y_{n-t}}{Y_n}.
\end{align*}
Hence, if we apply $\lim_{n\to\infty}$, we get
\begin{displaymath}
|\beta_i-c_{i}^{(s)}|\leq\epsilon\sum_{j,k=1}^N\sum_{t=0}^{\deg h_{ijk}/h}\left|[z^t]\frac{h_{ijk}}{h}\right|\zeta_s \rho^t=\epsilon\zeta_s\sum_{j,k=1}^N\left|\frac{h_{ijk}}{h}\right|(\rho).
\end{displaymath}
Thus, $|\beta_i-c_{i}^{(s)}|\to 0$ as $s\to\infty$. Hence, $\beta=\lim_{s\to\infty}C_s(\beta)$.
\end{proof}

From here, without special mention, we will assume an $h$-organized $Y$ and a $(Y,h)$-organized system.

\begin{prop}\label{tsBound}
Suppose that
\begin{itemize}
\item all the coefficients of $Y,g,h$ are nonnegative,
\item $\gamma\geq 0$, or $h(0)=0$ but $h$ is nonzero,
\item $f(\rho),Y(\rho)$ converge.
\end{itemize}
Then we have
\begin{displaymath}
0\leq\sqrt{(1-g(\rho))^2-4f(\rho)h(\rho)}-\gamma\leq \zeta_s\leq 1-\gamma-g(\rho)
\end{displaymath}
and
\begin{displaymath}
\lim_{s\to\infty}\zeta_s =\sqrt{(1-g(\rho))^2-4f(\rho)h(\rho)}-\gamma=1-\gamma-g(\rho)-2h(\rho)Y(\rho).
\end{displaymath}
Moreover, if $f$ has no singularity in $\{z\in\mathbb{C}\mid |z|<\rho+\epsilon\}$ for some $\epsilon>0$, then both $\gamma$ and the impurity, $\zeta_\infty$, are zero.
\end{prop}
\begin{proof}
First, we have $\zeta_s=1-\gamma-g(\rho)-2h(\rho)Y^{\leq s}(\rho)\leq 1-\gamma-g(\rho)$. Moreover, 
\begin{displaymath}
\zeta_s = \lim_{n\to\infty}\sum_{u=0}^{\deg h}h_u\sum_{v=s+1}^{n-u-s-1}\frac{Y_vY_{n-u-v}}{Y_n}
\end{displaymath}
gives $\zeta_s\geq 0$ always.

Now, if $h=0$, then $\gamma\geq0$, so we have $1-g(\rho)\geq\gamma\geq 0$. Hence,
\begin{displaymath}
0\leq\zeta_s=1-\gamma-g(\rho)=\sqrt{(1-g(\rho))^2-4f(\rho)h(\rho)}-\gamma.
\end{displaymath}

For the case $h$ is nonzero, then we have
\begin{displaymath}
Y(z)=\frac{1-g(z)-\sqrt{(1-g(z))^2-4f(z)h(z)}}{2h(z)}.
\end{displaymath}
Since $f(\rho),Y(\rho)$ converge, it gives
\begin{displaymath}
Y(\rho)=\frac{1-g(\rho)-\sqrt{(1-g(\rho))^2-4f(\rho)h(\rho)}}{2h(\rho)}.
\end{displaymath}
Hence,
\begin{displaymath}
\zeta_s = 1-g(\rho)-2h(\rho)Y^{\leq s}(\rho)-\gamma\geq1-g(\rho)-2h(\rho)Y(\rho)-\gamma
\end{displaymath}
and
\begin{displaymath}
0\leq\lim_{s\to\infty}\zeta_s = 1-g(\rho)-2h(\rho)Y(\rho)-\gamma=\sqrt{(1-g(\rho))^2-4f(\rho)h(\rho)}-\gamma.
\end{displaymath}

Now, consider the case that $f$ has no singularity in $\{z\in\mathbb{C}\mid |z| < \rho+\epsilon\}$. By \textbf{Theorem IV.7} in \cite{Fl}, $\rho$ is the closest singularity to zero of $Y$. If $h$ is zero, then 
\begin{displaymath}
Y(z)=\frac{f(z)}{1-g(z)}.
\end{displaymath}
Since $f$ has no singularity in $\{z\in\mathbb{C}\mid |z| < \rho+\epsilon\}$, it means $g(\rho)=1$. Hence, $\zeta_s=1-\gamma-g(\rho)=-\gamma\leq 0$, so the impurity and $\gamma$ are zero.

For the case that $h$ is nonzero, $f$ has no singularity in $\{z\in\mathbb{C}\mid |z|< \rho+\epsilon\}$ and
$g,h$ are polynomials, so
\begin{displaymath}
(1-g(\rho))^2-4f(\rho)h(\rho)=0.
\end{displaymath}
Then, $1-g(\rho)-2h(\rho)Y(\rho)=0$ and $\lim_{s\to\infty}\zeta_s=-\gamma$. Hence, it is enough to prove that $\gamma=0$. This can be induced from again \textbf{Theorem IV.7} in \cite{Fl}, which gives $\limsup (f_n)^{1/n} \leq \frac{1}{\rho+\epsilon}$.
\end{proof}

Combining \thmr{convergence} and \propr{tsBound}, we directly obtain the following.
\begin{thm}\label{linearEquation}
Suppose that 
\begin{itemize}
\item coefficients of $Y,g,h$ are nonnegative
\item $h$ is nonzero and $h(0)=0$
\item $f(\rho)$, $Y(\rho)$, $A_1(\rho)$, $\cdots$, $A_N(\rho)$ converge
\item the impurity is zero.
\end{itemize}
Then,
\begin{align*}
\beta_i = &\gamma_i + \sum_{j=1}^Ng_{ij}(\rho)\beta_j
+ \sum_{j,k=1}^Nh_{ijk}(\rho)(A_j(\rho)\beta_k+A_k(\rho)\beta_j).
\end{align*}
\end{thm}
This theorem is also a variation of Szeg\H{o}'s lemma. Moreover, this is linear on $\beta_j$'s when $A_j(\rho)$'s are given, and linear on $A_j(\rho)$'s when $\beta_j$'s are given. Also, if $A_j(\rho)$ are given and $\gamma_i$'s are zero, then it is a homogeneous linear system on $\beta_j$'s. For this case, we need additional conditions to solve completely, such as $\sum_{i=1}^N\beta_i = 1$. Lastly, it is remarkable that with this result, we can compute the density of tautology only with $\alpha_B$ values.

From the equation $Y(z)=f(z)+g(z)Y(z)+h(z)Y(z)^2$, $g$, $h$ represent the recursive structure and $f$ represents basic elements. Hence, it is natural to find an equation with $\gamma=0$ for given power series $Y(z)$.
\begin{defs}~
\begin{enumerate}[label=(\alph*)]
\item If $\gamma\not=1$, the $\gamma-\widehat{\gamma}$ conversion of the equation $Y(z)=f(z)+g(z)Y(z)+h(z)Y(z)^2$ is defined as $Y(z)=\widehat{f}(z)+\widehat{g}(z)Y(z)+\widehat{h}(z)Y(z)^2$ where
\begin{align*}
\widehat{f}(z)&=\frac{1-\widehat{\gamma}}{1-\gamma}f(z) + \frac{\widehat{\gamma}-\gamma}{1-\gamma}Y(z),\\
\widehat{g}(z)&=\frac{1-\widehat{\gamma}}{1-\gamma}g(z),\\
\widehat{h}(z)&=\frac{1-\widehat{\gamma}}{1-\gamma}h(z).
\end{align*}
\item If $\delta(z)$ is a polynomial, the $\delta$ conversion is defined as $Y(z)=\widetilde{f}(z)+\widetilde{g}(z)Y(z)+\widetilde{h}(z)Y(z)^2$ where
\begin{align*}
\widetilde{f}(z)&=f(z) + \delta(z)Y(z),\\
\widetilde{g}(z)&=g(z)-\delta(z),\\
\widetilde{h}(z)&=h(z).
\end{align*}
\end{enumerate}
\end{defs}

\begin{prop}~
\begin{enumerate}[label=(\alph*)]
\item For the $\gamma-\widehat{\gamma}$ conversion, we have
\begin{displaymath}
\lim_{n\to\infty}\frac{\widehat{f}_n}{Y_n}=\widehat{\gamma}.
\end{displaymath}
\item For the $\gamma-\widehat{\gamma}$ conversion, $\frac{\zeta_s}{1-\gamma}$ is invariant. Moreover, if $Y(\rho)$ converges, then $\frac{\zeta_\infty}{1-\gamma}$ is invariant.
\item For the $\delta$ conversion, we have
\begin{displaymath}
\widetilde{\gamma}=\gamma+\delta(\rho).
\end{displaymath}
\item For the $\delta$ conversion, $\zeta_s$ is invariant and if $Y(\rho)$ converges, then $\zeta_\infty$ is invariant.
\end{enumerate}
\end{prop}
\begin{proof}
(a),(c) are simple computation.

(b) We have
\begin{align*}
\widehat{\zeta}_s
&=1-\widehat{\gamma}-\widehat{g}(\rho)-2\widehat{h}(\rho)Y^{\leq s}(\rho)\\
&=\frac{1-\widehat{\gamma}}{1-\gamma}\left(1-\gamma-g(\rho)-2h(\rho)Y^{\leq s}(\rho)\right)\\
&=\frac{1-\widehat{\gamma}}{1-\gamma}\zeta_s.
\end{align*}

(d) We have
\begin{align*}
\widetilde{\zeta}_s
&=1-\widetilde{\gamma}-\widetilde{g}(\rho)-2\widetilde{h}(\rho)Y^{\leq s}(\rho)\\
&=1-\gamma-\delta(\rho)-g(\rho)+\delta(\rho)-2h(\rho)Y^{\leq s}(\rho)\\
&=\zeta_s.
\end{align*}
\end{proof}

Now, we are going to compute the numerical estimations of the ratio $\beta_i$'s by computing $s$-cut solutions, which means we expect that
\begin{displaymath}
\lim_{s\to\infty}\beta_{i}^{(s)}=\beta_i=\lim_{n\to\infty}\frac{A_{in}}{Y_n}
\end{displaymath}
is satisfied. Since the $s$-cut operator is quadratic, existence and uniqueness are not guaranteed. Hence, we will provide some conditions for existence, uniqueness and above convergence of the $s$-cut solution.

\begin{defs}
A $(Y,h)$-organized system $A_1,\cdots,A_N$ is a \textbf{natural partition of $Y$} if
\begin{align*}
Y(z) &= \sum_{i=1}^N A_i(z),\\
f(z) &= \sum_{i=1}^N f_i(z), \\
g(z) &= \sum_{i=1}^N g_{ij}(z), \\
2h(z) &= \sum_{i=1}^N(h_{ijk}(z) + h_{ikj}(z)).
\end{align*}
Also, a natural partition system is \textbf{nonnegative} if $Y$, $A_i$, $g$, $h$, $g_{ij}$'s $h_{ijk}$'s have no negative coefficients, and $\gamma, \gamma_i\geq 0$.
\end{defs}
Here, for a nonnegative natural partition system, we have $\gamma,\gamma_i\leq 1$ and $h_{ijk}(z)=c_{ijk}h(z)$ for some constant $c_{ijk}$.

\begin{prop}\label{FixedHyper}
Let $(c_1,\cdots,c_N)$ be a fixed point of the $s$-cut operator $C_s$ for a natural partition system
$(Y,A_1,\cdots,A_N)$ with nonzero $\zeta_s$. Then, $(c_1,\cdots,c_N)$ is on the hyperplane
$x_1+\cdots+x_N=\frac{\gamma}{\zeta_s}$ or $x_1+\cdots+x_N=1$ in $\mathbb{R}^N$.
\end{prop}
\begin{proof}
\begin{align*}
\sum_{i=1}^N c_i
=& \sum_{i=1}^N\gamma_i + \sum_{i=1}^N\sum_{j=1}^N g_{ij}(\rho)c_j
+ \sum_{i=1}^N\sum_{j,k=1}^Nh_{ijk}(\rho)(A_j^{\leq s}(\rho)c_k+A_k^{\leq s}(\rho)c_j) \\
&\qquad\qquad+ \sum_{i=1}^N\sum_{j,k=1}^N\zeta_s\frac{h_{ijk}}{h}(\rho)c_jc_k\\
=&\gamma + \sum_{j=1}^Ng(\rho)c_j
+\frac{1}{2}\sum_{j,k=1}^N\sum_{i=1}^N (h_{ijk}(\rho)+h_{ikj}(\rho))(A_j^{\leq s}(\rho)c_k+A_k^{\leq s}(\rho)c_j)\\
&\qquad\qquad+\frac{1}{2}\sum_{j,k=1}^N\sum_{i=1}^N\zeta_s(\frac{h_{ijk}}{h}(\rho)+\frac{h_{ikj}}{h}(\rho))c_jc_k\\
=&\gamma+g(\rho)\sum_{j=1}^N c_j
+\sum_{j,k=1}^N h(\rho)(A_j^{\leq s}(\rho)c_k+A_k^{\leq s}(\rho)c_j)
+\sum_{j,k=1}^N\zeta_sc_jc_k\\
=&\gamma+g(\rho)\sum_{j=1}^Nc_j + 2h(\rho)Y^{\leq s}(\rho)\sum_{j=1}^Nc_j + \zeta_s\left(\sum_{j=1}^Nc_j\right)^2.
\end{align*}
Hence,
\begin{displaymath}
(\zeta_s+\gamma)\sum_{j=1}^Nc_j = \gamma + \zeta_s\left(\sum_{j=1}^Nc_j\right)^2,
\end{displaymath}
which proves the proposition.
\end{proof}

\begin{prop}
The $s$-cut operator $C_s$ of a nonnegative natural partition system has a fixed point in
\begin{displaymath}
H:=\{(x_1,\cdots,x_N)\in\mathbb{R}^N: 0\leq x_i\leq 1,\sum_{i=1}^N x_i = 1\}.
\end{displaymath}
\end{prop}
\begin{proof}
Let $C_s(x_1,\cdots,x_N)=(c_1,\cdots,c_N)$. If $(x_1,\cdots,x_N)\in H$, then as in the proof of \propr{FixedHyper}, we have
\begin{displaymath}
\sum_{i=1}^Nc_i =\gamma+ g(\rho)\cdot 1 + 2h(\rho)Y^{\leq s}(\rho)\cdot 1 + \zeta_s\cdot 1^2 = 1.
\end{displaymath}
In the proof of \propr{tsBound}, we obtained $\zeta_s\geq 0$ from the fact that coefficients of $Y$ and $h$ are nonnegative. Hence, we have $c_i\geq 0$ for every $i$. Then, $\sum_{i=1}^Nc_i=1$ implies $c_i\leq 1$, so $(c_1,\cdots,c_N)\in H$.

Now, $H$ is a convex compact set in $\mathbb{R}^N$, so by the Brouwer fixed point theorem,
$C_s$ has a fixed point in $H$.
\end{proof}

By simple computation, we have
\begin{align*}
\frac{\partial c_i}{\partial x_j}
=&g_{ij}(r) + \sum_{k=1}^N(h_{ijk}(\rho)+h_{ikj}(\rho))A_k^{\leq s}(\rho) + \zeta_s\sum_{k=1}^N(\frac{h_{ijk}}{h}(\rho)+\frac{h_{ikj}}{h}(\rho))x_k.
\end{align*}
From this, we have the following result.
\begin{prop}
For the Jacobian $J$ of the $s$-cut operator $C_s$ of a nonnegative natural partition system,
\begin{displaymath}
\|J(x_1,\cdots,x_N)\|_1 = 1-\zeta_s-\gamma + 2\zeta_s\left(\sum_{i=1}^Nx_i\right)
\end{displaymath}
on $[0,\infty)^N$, where $\|\cdot\|_1$ denotes the $1$-norm of a matrix. In particular,
$\|J\|_1 = 1-\gamma+\zeta_s$ on $H$.
\end{prop}
\begin{proof}
Note that $\|B\|_1 = \max_{1\leq j\leq n}\sum_{i=1}^m|b_{ij}|$ for any $m\times n$ matrix $B$. Since the system is nonnegative, $\frac{\partial c_i}{\partial x_j}\geq 0$ on $[0,\infty)^N$. Then,
\begin{align*}
\sum_{i=1}^N\frac{\partial c_i}{\partial x_j}
 =& \sum_{i=1}^Ng_{ij}(\rho) + \sum_{k=1}^N\sum_{i=1}^N(h_{ijk}(\rho)+h_{ikj}(\rho))A_k^{\leq s}(\rho) \\
 &+ \zeta_s\sum_{k=1}^N\sum_{i=1}^N(\frac{h_{ijk}}{h}(\rho)+\frac{h_{ikj}}{h}(\rho))x_k\\
 =& g(\rho) + \sum_{k=1}^Nh(\rho)2A_{k}^{\leq s}(\rho) + \zeta_s\sum_{k=1}^N 2x_k\\
 =& g(\rho) + 2h(\rho)Y^{\leq s}(\rho) + 2\zeta_s\sum_{k=1}^N x_k\\
 =& 1-\zeta_s-\gamma+ 2\zeta_s\sum_{k=1}^N x_k,\\
\end{align*}
which proves the result.
\end{proof}

This result is also true when the nonnegativity condition is weakened: For instance, $\gamma$ and $\gamma_i$ may not be nonnegative. Moreover, we have the following result for general $p$-norms.
\begin{prop}
For the Jacobian $J$ of the $s$-cut operator $C_s$ of a natural partition system,
\begin{displaymath}
\|J(x_1,\cdots,x_N)\|_p \geq |1-\gamma+\zeta_s|
\end{displaymath}
on $H$. Note that $|1-\gamma+\zeta_s|=1-\gamma+\zeta_s$ when the given system is nonnegative, since we have $\gamma\leq 1$ and $\zeta_s\geq 0$.
\end{prop}
\begin{proof}
Let $J^T$ denote the transpose of the Jacobian. We have
\begin{displaymath}
J^T\begin{bmatrix}1 \\ \vdots \\ 1\end{bmatrix}=(1-\gamma+\zeta_s)\begin{bmatrix}1 \\ \vdots \\ 1\end{bmatrix}
\end{displaymath}
on $H$, from $\sum_{i=1}^N\frac{\partial c_i}{\partial x_j}=1-\zeta_s-\gamma+ 2\zeta_s\sum_{k=1}^N x_k = 1-\gamma +\zeta_s$. Hence, $1-\gamma+\zeta_s$ is an eigenvalue of $J^T$,
so is an eigenvalue of $J$. Thus, we get $\|J\|_p\geq |1-\gamma+\zeta_s|$.
\end{proof}

Since the norm of the Jacobian of the $s$-cut operator $C_s$ can be larger than 1, especially when $\gamma=0$, this fact may cause some convergence issues when we try to find an $s$-cut solution by applying the fixed point iteration method on $C_s$. Hence, we may consider the following modification.
\begin{defs}
The $\sigma$-shifted $s$-cut operator $\widetilde{C_s^\sigma}$ is defined as
\begin{displaymath}
\widetilde{C_s^\sigma}(x)=C_s(x)-\sigma\left(\sum_{i=1}^Nx_i-1\right)\cdot (1,1,\cdots,1).
\end{displaymath}
\end{defs}
Since $\widetilde{C_s^\sigma}(x)=C_s(x)$ for all $x\in H$, fixed points of $C_s$ on $H$ are fixed points of
$\widetilde{C_s^\sigma}$. Moreover,
\begin{displaymath}
\widetilde{J}=J-\begin{bmatrix}\sigma & \sigma & \cdots & \sigma \\\sigma &\ddots& \cdots &\vdots \\
\vdots & \vdots & \ddots & \vdots\\ \sigma & \cdots & \cdots & \sigma\end{bmatrix}=J-\sigma\mathbf{1},
\end{displaymath}
where $\widetilde{J}$ is the Jacobian for $\widetilde{C_s^\sigma}$.

From the Banach contraction principle, we deduce the following.
\begin{prop}
If the Jacobian $\widetilde{J}$ of $\widetilde{C_s^\sigma}$ satisfies $\|\widetilde{J}\|<1$ for a matrix norm $\|\cdot\|$ on $H$, then $\widetilde{C_s^\sigma}$ is a contraction on $H$, and $C_s$ has the unique fixed point on $H$.
\end{prop}
Note that since $H$ is compact, $\|\widetilde{J}\|<1$ is enough to apply the Banach contraction principle rather than the condition that there exists $K<1$ such that $\|\widetilde{J}\|\leq K$. From $\|A-B\|\geq | \|A\|-\|B\| |$, it would be best to choose $\sigma$ satisfying $\|J\|=\|\sigma\mathbf{1}\|$, and one of such choice is $\sigma=\frac{1-\gamma+\zeta_s}{N}$, which is from the 1-norm. Hence, we will call the $s$-cut operator shifted by this value as the \textbf{standard shifted $s$-cut operator}.
\begin{cor}
The Jacobian $\tilde{J}$ of the standard shifted $s$-cut operator $\widetilde{C_s}$ of a nonnegative natural partition system satisfies $\|\tilde{J}\|_1  < 1$ on $H$ if $1-2\gamma+2\zeta_s>0$ and
\begin{displaymath}
\frac{\partial c_i}{\partial x_j} < \frac{1-\gamma+\zeta_s}{N}+\max\left\{\frac{1-\gamma+\zeta_s}{N(1-2\gamma+2\zeta_s)},\frac{1}{2(N-1)}\right\}.
\end{displaymath}
Note that $1-2\gamma+2\zeta_s\leq 0$ implies $1-\gamma+\zeta_s\leq\frac{1}{2}$ , which means $\|J\|_1<1$ is already satisfied without shifting.
\end{cor}
\begin{proof}
Since $\|\widetilde{J}\|_1=\max\left\{\sum_{i=1}^N\left|\frac{\partial c_i}{\partial x_j}-\frac{1-\gamma+\zeta_s}{N}\right|\mid j=1,\cdots,N\right\}$, and we have $\sum_{i=1}^N\frac{\partial c_i}{\partial x_j}=1-\gamma+\zeta_s$ already, it is enough to prove that $a_i$s arranged as $\max\{\frac{1-\gamma+\zeta_s}{N(1-2\gamma+2\zeta_s)},\frac{1}{2(N-1)}\} + \frac{1-\gamma+\zeta_s}{N}> a_1\geq a_2\geq\cdots \geq a_m\geq\frac{1-\gamma+\zeta_s}{N}\geq a_{m+1}\geq\cdots\geq a_N\geq 0$ satisfying $\sum_{i=1}^Na_i=1-\gamma+\zeta_s$ satisfies $\sum_{i=1}^N|a_i-\frac{1-\gamma+\zeta_s}{N}|<1$. Since $\frac{1-\gamma+\zeta_s}{N}$ is the mean of $a_i$s, we may assume $m<N$. Easily,
\begin{align*}
\sum_{i=1}^N\left|a_i-\frac{1-\gamma+\zeta_s}{N}\right| &= \sum_{i=1}^m(a_i-\frac{1-\gamma+\zeta_s}{N}) + \sum_{i=m+1}^N(\frac{1-\gamma+\zeta_s}{N}-a_i)\\
&=\frac{1-\gamma+\zeta_s}{N}(N-2m) + \sum_{i=1}^ma_i - \sum_{i=m+1}^N a_i\\
&=\frac{1-\gamma+\zeta_s}{N}(N-2m) + 2\sum_{i=1}^ma_i - (1-\gamma+\zeta_s)\\
&= 2\sum_{i=1}^ma_i-\frac{2m}{N}(1-\gamma+\zeta_s).
\end{align*}
If $\frac{1}{2(N-1)}\geq\frac{1-\gamma+\zeta_s}{N(1-2\gamma+2\zeta_s)}$, we have $a_i < \frac{1}{2(N-1)} + \frac{1-\gamma+\zeta_s}{N}$, so
\begin{align*}
2\sum_{i=1}^ma_i&-\frac{2m}{N}(1-\gamma+\zeta_s)
\\&<2m(\frac{1}{2(N-1)}+\frac{1-\gamma+\zeta_s}{N})-\frac{2m}{N}(1-\gamma+\zeta_s)=\frac{m}{N-1}\leq 1.
\end{align*}
For the other case, if $m> N\left(1-\frac{1}{2(1-\gamma+\zeta_s)}\right)$, we have
\begin{align*}
2\sum_{i=1}^ma_i-\frac{2m}{N}(1-\gamma+\zeta_s)
&\leq 2\sum_{i=1}^Na_i - \frac{2m}{N}(1-\gamma+\zeta_s)\\
&\leq 2(1-\gamma+\zeta_s)-\frac{2m}{N}(1-\gamma+\zeta_s)\\
&=2(1-\gamma+\zeta_s)(1-\frac{m}{N})<1,
\end{align*}
and if $m\leq N(1-\frac{1}{2(1-\gamma+\zeta_s)})$, we have $a_i < \frac{1-\gamma+\zeta_s}{N(1-2\gamma+2\zeta_s)} + \frac{1-\gamma+\zeta_s}{N}$, so
\begin{align*}
2\sum_{i=1}^ma_i&-\frac{2m}{N}(1-\gamma+\zeta_s)
\\&<2m\left(\frac{1-\gamma+\zeta_s}{N}+\frac{1-\gamma+\zeta_s}{N(1-2\gamma+2\zeta_s)}\right)-\frac{2m}{N}(1-\gamma+\zeta_s)\\
&=\frac{m}{N}\frac{2(1-\gamma+\zeta_s)}{1-2\gamma+2\zeta_s}\\
&\leq\left(1-\frac{1}{2(1-\gamma+\zeta_s)}\right)\frac{2(1-\gamma+\zeta_s)}{1-2\gamma+2\zeta_s}=1.
\end{align*}
\end{proof}
This corollary gives a condition to have the unique $s$-cut solution by computing the 1-norm of the shifted $s$-cut operator. Finding the best choice to shift based on the matrix 1-norm of the Jacobian is equivalent to find $\sigma$ from given nonnegative sequences $a^{(1)},\cdots,a^{(N)}$ satisfying $\sum_i a_i^{(1)}=\cdots=\sum_i a_i^{(N)}$ such that minimizes the $\max\left\{\sum_{i}\left|a_i^{(j)}-\sigma\right|\mid j=1,\cdots, N\right\}$. For each $j$, it is well-known that the median minimizes $\sum_i\left|a_i^{(j)}-\sigma\right|$, compare with that mean minimizes $\sum_i(a_i^{(j)}-\sigma)^2$, where the standard shift operator is defined as to choose $\sigma$ as the mean, which is easier to compute than the median. Hence, it may be possible to refine the condition to have unique $s$-cut solution by considering the median rather than the mean. In such case, we may have to use some variant of the iteration method. For example, we may use different iteration functions for each iteration.

Lastly, we will prove the following, to understand $s$-cut solution as an approximation.
\begin{thm}
Suppose that a nonnegative natural partition system have a common contraction factor $K<1$ and a sequence of proper shifting factor $\{\sigma_s\}$ of
the $s$-cut operator $C_s$ satisfying $\left|\widetilde{C_s^{\sigma_s}}(x)-\widetilde{C_s^{\sigma_s}}(y)\right|\leq K|x-y|$ for any $x,y\in H$ except for finitely many $s$.

Then, there exists a sequence of $s$-cut solutions on $H$, $\beta^{(s)}=(\beta_1^{(s)},\cdots,\beta_N^{(s)})$, converging to $\beta=(\beta_1,\cdots,\beta_N)$ as $s\to\infty$.
\end{thm}
\begin{proof}
From \thmr{convergence}, $\lim_{s\to\infty}C_s(\beta)=\beta$ is satisfied. We may assume that $s$ is large enough to have a common contraction constant $K$. Then, we have
\begin{displaymath}
|\beta-\beta^{(s)}|\leq |\beta-C_s(\beta)|+|C_s(\beta)-\beta^{(s)}|=|\beta-C_s(\beta)|+|C_s(\beta)-C_s(\beta^{(s)})|.
\end{displaymath}
Since $C_s=\widetilde{C_s^\sigma}$ on $H$, $C_s$ is also a contraction on $H$ with the same contraction constant. Hence,
\begin{displaymath}
|\beta-\beta^{(s)}|\leq |\beta-C_s(\beta)|+|C_s(\beta)-C_s(\beta^{(s)})|\leq|\beta-C_s(\beta)|+K|\beta-\beta^{(s)}|.
\end{displaymath}
Thus,
\begin{displaymath}
|\beta-\beta^{(s)}|\leq\frac{1}{1-K}|\beta-C_s(\beta)|\to0
\end{displaymath}
as $s\to\infty$. Note that except for finitely many $s$'s, each $\beta^{(s)}$ is uniquely determined.
\end{proof}

The following table compares ratios at $s$ and $s$-cut solutions for the density of tautologies with one variable, which is $0.4232385...$.
\FloatBarrier
\begin{table}[h!]
\caption{Comparison of simple approximations and $s$-cuts for the one variable case}
\begin{center}
\begin{tabular}{|c|cc|}
\hline
 & ratio & $s$-cut\\
\hline
$s=10$ & 0.3101796...  & 0.4242620...\\
$s=100$ & 0.4187317... & 0.4232740...\\
$s=1000$ & 0.4227880... & 0.4232396...\\
$s=10000$ & 0.4231935... & 0.4232386...\\
\hline
\end{tabular}
\end{center}
\end{table}
\FloatBarrier
Hence, the $s$-cut solution seems to work as a memory-time trade-off to compute an approximation.

\section{Asymptotic Behavior}\label{asymptotic}

Even though we have a general method to compute the exact density of tautology, it is hard to compute for the case $m\geq 5$ since we need at least $2^{2^m}$ computations and memories. Hence, to estimate its asymptotic behavior, we require some different approach.

At first, fix $m$-element variable set $X$ as the set of variables $\{x_0,x_1,\cdots,x_{m-1}\}$, so it will generate a chain structure as the number of variable is changed.
\begin{defs}
Let $\overline{X}=\{x_0,x_1,\cdots\}$ be a countably infinite set of variables, and $\overline{\mathcal{W}}$ be the set of well-formed formulae of $\overline{X}$. For any $\sigma\in S_{\{0,1,2,\dots\}}=:S_\infty$, the set of all permutations of $\{0,1,2\dots\}$ with a finite support, we have a natural action on $\overline{\mathcal{W}}$ defined as
\begin{align*}
\sigma x_i &= x_{\sigma(i)},\\
\sigma \neg\phi &= \neg\sigma\phi,\\
\sigma[\phi\to\psi] &= \sigma\phi\to\sigma\psi.
\end{align*}
A formula $\phi\in\overline{\mathcal{W}}$ is a \textbf{type formula} if for every occurrence of $x_i$, there must exist occurrences of $x_0,\cdots,x_{i-1}$ before it. The \textbf{type of a well-formed formula} $\psi$ is the type formula $\phi$ such that there exists $\sigma\in S_\infty$ satisfying $\psi=\sigma\phi$. It is easy to prove that the type of a well-formed formula exists uniquely.

Lastly, for any formula $\phi\in\overline{\mathcal{W}}$, define an estimator as
\begin{displaymath}
\nu(\phi) := |\overline{X}[\phi]| - \frac{1}{2}\ell(\phi).
\end{displaymath}
\end{defs}

\begin{prop}
For any $\sigma\in S_\infty$, we have $F_{\sigma\phi}=\{\sigma v\mid v\in F_\phi\}$ where $(\sigma v)(x_i)=v(\sigma^{-1} x_i)$. In particular, $\phi$ is a tautology or an antilogy if and only if its type is a tautology or an antilogy, respectively.
\end{prop}

\begin{prop}\label{typeValue}
For any type formula $\phi$ and $m\geq |\overline{X}[\phi]|$, we have
\begin{displaymath}
\left.\left(\sum_{\psi\textrm{:$\phi$-type}, \overline{X}[\psi]\subseteq\{x_0,\cdots,x_{m-1}\}}z^{\ell(\psi)}\right)\right|_{z=\frac{1}{2\sqrt{m}+1}}=\frac{m^{\underline{|\overline{X}[\phi]|}}}{(2\sqrt{m}+1)^{\ell(\phi)}}=\Theta(m^{\nu(\phi)}).
\end{displaymath}
where $m^{\underline{k}}$ is the falling factorial $m(m-1)\cdots(m-k+1)$.
\end{prop}
From \thmr{linearEquation}, we have a relation between generating function values at the singularity point and limit ratios of coefficients, so it can be expected that tautologies with large $\nu$ values dominate the density of tautologies. Since $W(\frac{1}{2\sqrt{m}+1})=\sqrt{m}$, the density is expected as $m^{\nu(\phi)-\frac{1}{2}}$.

Now, we will prove basic properties of $\nu$. We begin with the following lemma.
\begin{lem}\label{lemTaut}
\begin{enumerate}[label=(\alph*)]
\item If $\phi$ has no $\neg$'s, then $\phi$ is true when the rightmost variable is true.
\item If $\phi$ has no repeated variables, then there is a truth assignment that makes $\phi$ true and a truth assignment that makes $\phi$ false.
\item Suppose that $\phi$ has no repeated variables, no $\neg$'s and that $p$ is a variable in $\phi$ but $\phi\not=p$. Then, there is a truth assignment that $p$ is true and a truth assignment that $p$ is false, where both of them make $\phi$ true.
\item Suppose that $\phi$ has no repeated variables, no $\neg$'s and that $p$ is not the rightmost variable of $\phi$. Then, there is a truth assignment that makes $\phi$ false and $p$ is true.
\end{enumerate}
\end{lem}
\begin{proof}
(a) We will induct on the length of $\phi$. If $\phi$ is a variable, then done. Otherwise, since $\phi=\psi\to\eta$ and $\eta$ is true by the induction hypothesis, so is $\phi$.

(b) By induction, if $\phi$ is a variable, then done. Otherwise, it is trivial when $\phi=\neg\psi$ for some $\psi$, since $\neg$ reverses trueness and falseness. Now, if $\phi=\psi\to\eta$, then there are an assignment $u$ on $\psi$ that makes $\psi$ true and an assignment $v$ on $\eta$ that makes $\eta$ false, by induction hypothesis. Since $\phi$ has no repeated variables, $X[\psi]\cap X[\eta]=\emptyset$. Thus, we have an assignment $u\oplus v$ such that $u\oplus v=u$ on $X[\psi]$, $u\oplus v=v$ on $X[\eta]$. Precisely, with the convention $\mathcal{VA}=\mathcal{P}(X)$, we may define $u\oplus v=u[\psi]\cup v[\eta]$. Then, $\llbracket\phi;u\oplus v\rrbracket=1-\llbracket\psi;u\rrbracket(1-\llbracket\eta;v\rrbracket)=0$. Hence, $u\oplus v$ is an assignment that makes $\phi$ false. Similarly, there is an assignment that makes $\phi$ true.

(c) Since $\phi$ has no $\neg$'s and $\phi$ is not $p$, $\phi=\psi\to\eta$ for some $\psi,\eta$. If $p\in X[\psi]$, then there exists an assignment $u$ that makes $\eta$ true by (b). Here, $p\not\in X[\eta]$, so we may choose $u^{-}$, $u^+$ such that $u^-=u=u^+$ on $X[\eta]$ and $u^-(p)=0$, $u^+(p)=1$. Similarly, when $p\in X[\eta]$, it can be done by using an assignment that makes $\psi$ false.

(d) Since $\phi$ has no $\neg$'s and $p$ is not the rightmost variable, clearly we have $\phi\not=p$. If $p\not\in X[\phi]$, then the case (b) is applicable. So we only need to consider the case that $\phi=\psi\to\eta$ for some $\psi$ and $\eta$. First, suppose that $p\in X[\psi]$. Then, by (b), there is an assignment $u$ that makes $\eta$ false. Now, if $\psi\not=p$, then by (c), there is an assignment $v$ that makes both $p$ and $\psi$ true. Then, we may construct $u\oplus v$ so $p$ is true and $\phi$ is false, since $\phi$ has no repeated variables. If $\psi=p$, then for any assignment $u$ that makes $\eta$ false, we may construct $u^+$ makes $p$ is true, so $\phi$ is false.

Now, if $p\in X[\eta]$, then by the induction hypothesis, there is an assignment $u$ that $\eta$ is false but $p$ is true. By (b), there is an assignment $v$ that makes $\psi$ true, so there is an assignment $u\oplus v$ that makes $\phi$ false.
\end{proof}

Then, we have the following.
\begin{prop}\label{wffnorm}~
\begin{enumerate}[label=(\alph*)]
\item For any well-formed formula $\phi$, $\nu(\phi)\leq\frac{1}{2}$.
\item For any tautology $\phi$, $\nu(\phi)\leq-\frac{1}{2}$. Moreover, $\nu(\phi)=-\frac{1}{2}$ if and only if $\phi$ does not contain $\neg$ symbol, $\phi$ has unique variable appears twice, and every other variables in $\phi$ appears only once.
\item For any antilogy $\phi$, $\nu(\phi)\leq-1$.
\end{enumerate}
\end{prop}
\begin{proof}(a)
Induction on the length. At first, $\nu(x_i)=1-\frac{1}{2}=\frac{1}{2}$, and $\nu(\neg\phi)=\nu(\phi)-\frac{1}{2}\leq 0$. Finally, we have
\begin{displaymath}
\nu(\phi\to\psi)\leq |X[\phi]|+|X[\psi]| - \frac{1}{2}(\ell(\phi)+\ell(\psi)+1)=\nu(\phi)+\nu(\psi)-\frac{1}{2}\leq\frac{1}{2}.
\end{displaymath}

(b) For a well-formed formula $\phi$, the number of occurrences of variables is exactly the number of occurrences of $\to$'s plus 1. Let $R$ be the number of variables in $\phi$ that do not appear first time in $\phi$, $Y$ be the number of occurrences of $\to$'s, and $N$ be the number of occurrences of $\neg$'s. Then, we have
\begin{align*}
\nu(\phi) &= |X[\phi]|-\frac{1}{2}\ell(\phi)\\
&=Y + 1 - R-\frac{1}{2}(Y + (Y+1) + N)\\
&=\frac{1}{2}-R-\frac{1}{2}N.
\end{align*}
Hence, $\nu(\phi)\geq 0$ implies $R=0$, so $\phi$ has no repeated variables. Then, by \lemr{lemTaut}(b), $\phi$ is not a tautology. The remaining part follows from the fact that $R\geq 1$ and $\nu(\phi)=-\frac{1}{2}$ imply $N=0$.

(c) By \lemr{lemTaut}(a) and (b), any antilogy $\phi$ needs at least one $\neg$ and repeated variables. Hence, $R\geq 1$ and $N\geq 1$, so $\nu(\phi)\leq -1$.
\end{proof}

\begin{prop} Suppose that $\phi$ is a tautology and $\nu(\phi)=-\frac{1}{2}$. Then, there are well-formed formulae $\psi_1,\cdots,\psi_k,\eta$ without $\neg$'s, pairwise common variables, and repeated variables such that $\phi$ is $\psi_1\to[\psi_2\to[\cdots\to[\psi_k\to [p\to\eta]]\cdots ]]$ where $p$ is the rightmost variable of $\eta$. Here, $k=0$ is possible.
\end{prop}
\begin{proof}
First, $\phi$ has no $\neg$'s and has the unique repeated variable $p$ which appears twice, by above proposition. Hence, $\phi=\psi\to\eta$ for some $\psi,\eta$.

Suppose $\psi$ and $\eta$ have no common variables. Then, by \lemr{lemTaut}(a), there is an assignment $u$ that makes $\psi$ true. Hence, if there is an assignment $v$ that makes $\eta$ false, we have $u\oplus v$ that makes $\phi=\psi\to\eta$ false. Thus, there is no assignment that makes $\eta$ false, so $\eta$ is again a tautology. This implies that $p\in X[\eta]$, since every tautology has at least one repeated variable. Hence, $\eta$ is again a tautology with $|\eta|=-\frac{1}{2}$. Then, by induction on length, $\eta$ is $\psi_2\to[\cdots\to[\psi_k\to [p\to\eta']]\cdots ]$ and so, $\phi$ is $\psi_1\to[\psi_2\to[\cdots\to[\psi_k\to [p\to\eta']]\cdots ]]$. Thus, done.

Now, assume that $\psi$ and $\eta$ have a common variable. Then, from the uniqueness of the repeated variable of $\phi$, it must be $p$. If $\psi\not=p$, then by \lemr{lemTaut}(b), there is an assignment $u$ that makes $\eta$ false. If $u(p)=1$, then by \lemr{lemTaut}(c), there is an assignment $v$ on $\psi$ that makes both $p$ and $\psi$ true. Also, if $u(p)=0$, then we have an assignment $v$ on $\psi$ that makes $\psi$ true and $p$ false. In any case, $u\oplus v$ makes $\phi=\psi\to\eta$ false, which is a contradiction. So $\psi=p$.

Then, we have $\phi=p\to\eta$. If $p$ is not the rightmost variable of $\eta$, then by \lemr{lemTaut}(d), there is an assignment $u$ on $\eta$ that makes $\eta$ false and $u(p)=1$. Hence, $u$ makes $\phi$ false, which is a contradiction. Thus, $p$ is the rightmost variable of $\eta$.
\end{proof}

From these results, we may guess that the density of tautologies is of $\frac{1}{m}$ order: since the maximum $\nu(\cdot)$ of well-formed formulae is $\frac{1}{2}$ and the maximum $\nu(\cdot)$ of tautologies is $-\frac{1}{2}$, we may expect $\frac{m^{-\frac{1}{2}}}{m^{\frac{1}{2}}}=\frac{1}{m}$ order. Similarly, for antilogies, we may expect $\frac{1}{m\sqrt{m}}$ order.

Also, this result is very similar to the definition of simple tautologies defined in \cite{Fo} and \cite{Za2}, which give that in the logic system with $\to$ and negative literals the density of tautologies is asymptotically same as the density of simple tautologies, i.e., $\frac{7}{8m}+O(\frac{1}{m^2})$. In \cite{Fo}, a simple tautology is defiend as a tautology of the form
\begin{displaymath}
\phi_1\to[\phi_2\to[\cdots\to[\phi_n\to p]\cdots ]],
\end{displaymath}
which can be simplified with the canonical form of an expression, also defined in \cite{Fo}, as 
\begin{displaymath}
\phi_1,\cdots,\phi_n\mapsto p
\end{displaymath}
where each $\phi_i$ is a well-formed formula and $p$ is a variable, with condition $\phi_i=p$ for some $i$; or for some distinct pair $i$ and $j$, $\phi_i$ is a variable and $\phi_j=\bar{\phi_i}$. Here, $\bar{x}$ means negative literal of $x$.

The former is called a simple tautology of the first kind, and the latter is called a simple tautology of the second kind. But there are some differences between our case and the given cases. Firstly, for the case of implication with negative literals, there is no antilogy. Secondly, we have to negate, rather than using negative literals, which increases the length of the formula. It introduces the factor $\sqrt{m}$ in asymptotic ratio, which changes the order. Hence, we will define simple tautologies for our propositional logic system as follows.

\begin{defs} In the following, $k\geq 1$.
\begin{enumerate}[label=(\alph*)]
\item A well-formed formula $\phi$ is a \textbf{simple tautology}, if there exist well-formed formulae $\psi_1,\cdots,\psi_k$ and a variable $p$ such that $\phi$ is
\begin{displaymath}
\psi_1,\cdots,\psi_k\mapsto p
\end{displaymath}
with $\psi_i=p$ for some $i$. Let $\mathcal{S}_1$ be the set of simple tautologies.

\item A well-formed formula $\psi_1,\cdots,\psi_k\mapsto p$ is a \textbf{strict simple tautology}, if $\psi_1=p$ and $\psi_2,\cdots,\psi_k\not=p$. Let $\mathcal{S}_c$ be the set of strict simple tautologies.
\end{enumerate}
\end{defs}

Now, we will try to analyze the asymptotic behavior for the density of tautologies. Actually, it is easy to prove that $\Omega(\frac{1}{m})$. For any tautology $\psi$, $\neg\neg\psi$ is a tautology and for any well-formed formula $\phi$, $p\to[\phi\to p]$ is a tautology for any variable $p$. Hence, we have
\begin{displaymath}
[z^n]W_\emptyset(z) \geq [z^{n-2}]W_\emptyset(z) + m\cdot([z^{n-4}]W(z)).
\end{displaymath}
Then, we have
\begin{displaymath}
[z^n]W(z)\simeq\sqrt{\frac{2m+\sqrt{m}}{4\pi n^3}}(2\sqrt{m}+1)^n
\end{displaymath}
from the expression
\begin{align*}
W(z)&=\frac{1-z-\sqrt{(1-(2\sqrt{m}+1)z)(1+(2\sqrt{m}-1)z)}}{2z}\\
&=\sqrt{m}-\sqrt{2m+\sqrt{m}}\sqrt{1-(2\sqrt{m}+1)z} + O(1-(2\sqrt{m}+1)z).
\end{align*}
Hence, $\lim_{n\to\infty}\frac{[z^{n-2}]W(z)}{[z^n]W(z)}=\frac{1}{(2\sqrt{m}+1)^2}$ and $\lim_{n\to\infty}\frac{[z^{n-4}]W(z)}{[z^n]W(z)}=\frac{1}{(2\sqrt{m}+1)^4}$. Thus,
\begin{displaymath}
\lim_{n\to\infty}\frac{[z^n]W_\emptyset(z)}{[z^n]W(z)}\geq\left(\frac{1}{(2\sqrt{m}+1)^2}\lim_{n\to\infty}\frac{[z^{n-2}]W_\emptyset(z)}{[z^{n-2}]W(z)}\right)+\frac{m}{(2\sqrt{m}+1)^4}
\end{displaymath}
and so,
\begin{displaymath}
\lim_{n\to\infty}\frac{[z^n]W_\emptyset(z)}{[z^n]W(z)}\geq\frac{\sqrt{m}}{4(\sqrt{m}+1)(2\sqrt{m}+1)^2},
\end{displaymath}
where we have $\frac{\sqrt{m}}{4(\sqrt{m}+1)(2\sqrt{m}+1)^2}=\Theta(\frac{1}{m})$. Thus, the density of tautologies is $\Omega(\frac{1}{m})$.

For antilogies, we can get
$\Omega(\frac{1}{m\sqrt{m}})$ from
\begin{displaymath}
[z^n]W_{\mathcal{VA}}(z)\geq [z^{n-1}]W_\emptyset(z),
\end{displaymath}
since every $\neg\phi$ is an antilogy for any tautology $\phi$.

Also, simple tautologies give better lower bound as follows.
\begin{prop}\label{S1gen}~
\begin{enumerate}[label=(\alph*)]
\item The generating function $S_1$ of $\mathcal{S}_1$ is
\begin{displaymath}
\frac{mz^3}{(1+z^2-zW(z))(1-zW(z))}
\end{displaymath}
and
\begin{displaymath}
\lim_{n\to\infty}\frac{[z^n]S_1(z)}{[z^n]W(z)}=\frac{m(4m+6\sqrt{m}+3)}{(\sqrt{m}+1)^2(2m+3\sqrt{m}+2)^2}=\frac{1}{m}-\frac{7}{2m\sqrt{m}} + \frac{7}{m^2} + O(\frac{1}{m^2\sqrt{m}}).
\end{displaymath}
\item The generating function $S_c$ of $\mathcal{S}_c$ is
\begin{displaymath}
\frac{mz^3}{1+z^2-zW(z)}
\end{displaymath}
and
\begin{displaymath}
\lim_{n\to\infty}\frac{[z^n]S_c(z)}{[z^n]W(z)}=\frac{m}{(2m+3\sqrt{m}+2)^2}=\frac{1}{4m}-\frac{3}{4m\sqrt{m}} + \frac{19}{16m^2} + O(\frac{1}{m^2\sqrt{m}}).
\end{displaymath}
\end{enumerate}
\end{prop}
\begin{proof}
(a) The generating function of well-formed formulae of the form $\psi_1,\cdots,\psi_k\mapsto p$ is
\begin{displaymath}
mz(zW(z)) + mz(zW(z))^2 + mz(zW(z))^3 + \cdots = \frac{mz^2W(z)}{1-zW(z)}.
\end{displaymath}
Here, $mz$ term is for the variable $p$, and $zW(z)$ term is for the $\psi_i$ with $\to$ symbol. Now we select simple tautologies by using the fact that
a given well-formed formula is not a simple tautology if and only if every $\psi_i$ is not $p$. We induce that the generating function of such well-formed formulae is
\begin{displaymath}
mz(z(W(z)-z)) + mz(z(W(z)-z))^2 + mz(z(W(z)-z))^3 + \cdots = \frac{mz^2(W(z)-z)}{1+z^2-zW(z)}.
\end{displaymath}
Hence, we have
\begin{displaymath}
S_1(z)=\frac{mz^2W(z)}{1-zW(z)}-\frac{mz^2(W(z)-z)}{1+z^2-zW(z)}=\frac{mz^3}{(1+z^2-zW(z))(1-zW(z))}.
\end{displaymath}
Then, by Szeg\H{o}'s lemma, when we take $\rho_0=\frac{1}{2\sqrt{m}+1}$, we have
\begin{displaymath}
\lim_{n\to\infty}\frac{[z^n]S_1(z)}{[z^n]W(z)}=\frac{\lim_{z\to \rho_0^{-}}S_1'(z)\sqrt{1-z/\rho_0}}{\lim_{z\to \rho_0^{-}}W'(z)\sqrt{1-z/\rho_0}}.
\end{displaymath}
Now, we have
\begin{align*}
S_1'(z)&=-\frac{mz^3((1+z^2-zW(z))(-zW'(z)) + (1-zW(z))(-zW'(z))}{(1+z^2-zW(z))^2(1-zW(z))^2} + R(z)\\
&=\frac{mz^4(2+z^2-2zW(z))W'(z)}{(1+z^2-zW(z))^2(1-zW(z))^2} + \widetilde{R}(z)
\end{align*}
where $\lim_{z\to \rho_0^{-}}R(z)\sqrt{1-z/\rho_0}=\lim_{z\to \rho_0^{-}}\widetilde{R}(z)\sqrt{1-z/\rho_0}=0$. Thus, we have
\begin{displaymath}
\lim_{n\to\infty}\frac{[z^n]S_1(z)}{[z^n]W(z)}=\frac{m\rho_0^4(2+\rho_0^2-2\rho_0W(\rho_0))}{(1+\rho_0^2-\rho_0W(\rho_0))^2(1-\rho_0W(\rho_0))^2}=\frac{m(4m+6\sqrt{m}+3)}{(\sqrt{m}+1)^2(2m+3\sqrt{m}+2)^2}.
\end{displaymath}
Also, from $S_1(z)(1+z^2-zW(z))(1-zW(z))=mz^3$ and $z(W(z))^2=W(z)-mz-zW(z)$, we have an equation
\begin{displaymath}
S_1(z)=mz^3+(m-1)z^2S_1(z)+z(1+z+z^2)W(z)S_1(z),
\end{displaymath}
and if we use \thmr{linearEquation}, we get
\begin{displaymath}
\lim_{n\to\infty}\frac{[z^n]S_1(z)}{[z^n]W(z)}=\frac{\rho(1+\rho_0+\rho_0^2)S_1(\rho_0)}{1-(m-1)\rho_0^2-\rho_0(1+\rho_0+\rho_0)^2W(\rho_0)}=\frac{m(4m+6\sqrt{m}+3)}{(\sqrt{m}+1)^2(2m+3\sqrt{m}+2)^2}.
\end{displaymath}
which matches to the result from the Szeg\H{o}'s lemma.

(b) This can be done similarly as (a).
\end{proof}

Now, to improve more, we will consider the following.
\begin{defs} Let $\mathcal{B}$ be a set of tautologies and antilogies.
\begin{enumerate}[label=(\alph*)]
\item The strong $\mathcal{B}$-category is a partition of well-formed formulae consisting of strong $\mathcal{B}$-tautologies ($\mathcal{T}_*$), $\mathcal{B}$-antilogies ($\mathcal{A}_*$), and $\mathcal{B}$-unknowns ($\mathcal{U}_*$) determined by $\mathcal{B}$ such that
\begin{itemize}
\item $\phi\in\mathcal{T}_*$ if and only if $\phi\in\mathcal{B}$ and $\phi$ is a tautology; $\phi$ is $\neg\psi$ form where $\psi\in\mathcal{A}_*$; or $\phi$ is $\psi\to\eta$ form where $\eta\in\mathcal{T}_*$.
\item $\phi\in\mathcal{A}_*$ if and only if $\phi\in\mathcal{B}$ and $\phi$ is an antilogy; $\phi$ is $\neg\psi$ form where $\psi\in\mathcal{T}_*$; or $\phi$ is $\psi\to\eta$ form where $\psi\in\mathcal{T}_*$ and $\eta\in\mathcal{A}_*$.
\item $\phi\in\mathcal{U}_*$ if and only if $\phi\not\in\mathcal{T}_*\cup\mathcal{A}_*$.
\end{itemize}
The following table shows this recursive classification.
\FloatBarrier
\begin{table}[h!]
\caption{Strong categories}
\begin{center}
\begin{tabular}{|c|c|c|c|}
 & $\mathcal{T}_*$ & $\mathcal{U}_*$ & $\mathcal{A}_*$\\
 \hline
$\neg$ & $\mathcal{A}_*$ & $\mathcal{U}_*$ & $\mathcal{T}_*$ \\
$\mathcal{T}_*\to$ & $\mathcal{T}_*$ & $\mathcal{U}_*$ & $\mathcal{A}_*$ \\
$\mathcal{U}_*\to$ & $\mathcal{T}_*$ & $\mathcal{U}_*$ & $\mathcal{U}_*$ \\
$\mathcal{A}_*\to$ & $\mathcal{T}_*$ & $\mathcal{U}_*$ & $\mathcal{U}_*$
\end{tabular}
\end{center}
\end{table}
\FloatBarrier

\item  The weak $\mathcal{B}$-category is a partition of well-formed formulae consisting of strong $\mathcal{B}$-tautologies ($\mathcal{T}^*$), $\mathcal{B}$-unknowns ($\mathcal{U}^*$), and $\mathcal{B}$-antilogies ($\mathcal{A}^*$) determined by $\mathcal{B}$ such that
\begin{itemize}
\item $\phi\in\mathcal{T}^*$ if and only if $\phi\in\mathcal{B}$ and $\phi$ is a tautology; $\phi$ is $\neg\psi$ form where $\psi\in\mathcal{A}^*$; or $\phi$ is $\psi\to\eta$ form where $\eta\in\mathcal{T}^*$ or $\psi\in\mathcal{A}^*$.
\item $\phi\in\mathcal{A}^*$ if and only if $\phi\in\mathcal{B}$ and $\phi$ is an antilogy; $\phi$ is $\neg\psi$ form where $\psi\in\mathcal{T}^*$; or $\phi$ is $\psi\to\eta$ form where $\psi\in\mathcal{T}^*$ and $\eta\in\mathcal{A}^*$.
\item $\phi\in\mathcal{U}^*$ if and only if $\phi\not\in\mathcal{T}^*\cup\mathcal{A}^*$.
\end{itemize}
The following table shows this recursive classification.
\FloatBarrier
\begin{table}[h!]
\caption{Weak categories}
\begin{center}
\begin{tabular}{|c|c|c|c|}
 & $\mathcal{T}^*$ & $\mathcal{U}^*$ & $\mathcal{A}^*$\\
 \hline
$\neg$ & $\mathcal{A}^*$ & $\mathcal{U}^*$ & $\mathcal{T}^*$ \\
$\mathcal{T}^*\to$ & $\mathcal{T}^*$ & $\mathcal{U}^*$ & $\mathcal{A}^*$ \\
$\mathcal{U}^*\to$ & $\mathcal{T}^*$ & $\mathcal{U}^*$ & $\mathcal{U}^*$ \\
$\mathcal{A}^*\to$ & $\mathcal{T}^*$ & $\mathcal{T}^*$ & $\mathcal{T}^*$
\end{tabular}
\end{center}
\end{table}
\FloatBarrier

\item A well-formed formula $\phi$ is weak (resp. strong) $\mathcal{B}$-basic if $\phi$ is not a weak (resp. strong) $\mathcal{B}$-unknown and $\phi$ is a weak (resp. strong) $(\mathcal{B}\setminus\{\phi\})$-unknown.

\item The set $\mathcal{B}$ is weak (resp. strong) basic if every $\phi\in\mathcal{B}$ is weak (resp. strong) $\mathcal{B}$-basic.

\item A set of $\mathcal{B}$-knowns $\mathcal{B}'$ is a weak (resp. strong) basis of $\mathcal{B}$ if $\mathcal{B}'$ is weak (resp. strong) basic and every $\phi\in\mathcal{B}$ is a weak (resp. strong) $\mathcal{B}'$-known.
\end{enumerate}
\end{defs}

This proposition is true for both weak and strong category.
\begin{prop} Suppose $\mathcal{B}$ is a set of tautologies and antilogies.
\begin{enumerate}[label=(\alph*)]
\item For a basis $\widetilde{\mathcal{B}}$ of $\mathcal{B}$, $\widetilde{\mathcal{B}}$-category is same as $\mathcal{B}$-category.
\item Every well-formed formula $\phi$ of a basis $\widetilde{\mathcal{B}}$ of $\mathcal{B}$ is $\mathcal{B}$-basic.
\item Every $\mathcal{B}$ has a basis and a well-formed formula $\phi$ is in a basis $\widetilde{\mathcal{B}}$ if and only if $\phi$ is $\mathcal{B}$-basic. In particular, there is a unique basis $\widetilde{\mathcal{B}}$ of $\mathcal{B}$, which is a subset of $\mathcal{B}$.
\end{enumerate}
\end{prop}
\begin{proof}(a) With induction on the length of well-formed formulae, it comes from the recursive structure of categories.

(b) If $\phi$ is not $\mathcal{B}$-basic, then $\phi$ is $(\mathcal{B}\setminus\{\phi\})$-known. Now, for every well-formed formula $\psi$ with $\ell(\psi)<\ell(\phi)$, $\mathcal{B}$-category, $(\mathcal{B}\setminus\{\phi\})$-category, $\widetilde{\mathcal{B}}$-cateogory and $(\widetilde{\mathcal{B}}\setminus\{\phi\})$-category are all same. Hence, $\phi$ is a $(\widetilde{\mathcal{B}}\setminus\{\phi\})$-known, contradicting that $\widetilde{\mathcal{B}}$ is basic.

(c) It is enough to show that every $\mathcal{B}$-basic $\phi$ is in $\widetilde{\mathcal{B}}$ and the set of $\mathcal{B}$-basic well-formed formulae is a basis. If $\phi$ is $\mathcal{B}$-basic, then $\phi$ is not a $(\mathcal{B}\setminus\{\phi\})$-known and so, not a $(\widetilde{\mathcal{B}}\setminus\{\phi\})$-known. Since $\widetilde{\mathcal{B}}$ is a basis, $\phi$ is a $\widetilde{\mathcal{B}}$-known, and so, $\phi\in\widetilde{\mathcal{B}}$.

Let $\widehat{\mathcal{B}}$ be the set of $\mathcal{B}$-basic well-formed formulae. Then, by the definition of basic well-formed formula, $\widehat{\mathcal{B}}\subseteq\mathcal{B}$. Since $\phi\in\widehat{\mathcal{B}}$ is not a $(\mathcal{B}\setminus\{\phi\})$-known, it is not a $(\widehat{\mathcal{B}}\setminus\{\phi\})$-known, and so, $\widehat{\mathcal{B}}$ is basic.
Let $\psi$ be a shortest $\mathcal{B}$-known that is not a $\widehat{\mathcal{B}}$-known. Then, for every shorter well-formed formula $\psi'$ than $\psi$, $(\mathcal{B}\setminus\{\psi\})$-category is same as $\widehat{\mathcal{B}}$-category. Now, $\psi$ is not $\mathcal{B}$-basic, so is a $(\mathcal{B}\setminus\{\psi\})$-known, and hence $\psi$ is a $\widehat{\mathcal{B}}$-known, which is a contradiction. So every $\mathcal{B}$-known is a $\widehat{\mathcal{B}}$-known, and so, $\widehat{\mathcal{B}}$ is a basis of $\mathcal{B}$.
\end{proof}

\begin{prop}
For a set $\mathcal{B}$ of tautologies and antilogies, every strong $\mathcal{B}$-tautology is a weak $\mathcal{B}$-tautology, and every weak $\mathcal{B}$-tautology is a tautology. Hence, every weak $\mathcal{B}$-basic well-formed formula is strong $\mathcal{B}$-basic.
\end{prop}

The following system of equations naturally follows from the structure of $\mathcal{B}$-categories.

\begin{prop} Let $\mathcal{B}$ be a set of tautologies and antilogies.
\begin{enumerate}[label=(\alph*)]
\item Let $BT_*, BA_*, T_*, U_*, A_*$ be the generating functions of the strong basic tautologies of $\mathcal{B}$, strong basic antilogies of $\mathcal{B}$, strong $\mathcal{B}$-tautologies, strong $\mathcal{B}$-unknowns, and strong $\mathcal{B}$-antilogies, respectively. Then the following system of equations is satisfied.
\begin{align*}
T_*(z) &= BT_*(z) + zA_*(z) + zT_*(z)W(z),\\
U_*(z) &= mz - BT_*(z) - BA_*(z) + zU_*(z) + z[U_*(z)W(z)+A_*(z)W(z)-A_*(z)T_*(z)],\\
A_*(z) &= BA_*(z) + zT_*(z) + zA_*(z)T_*(z).
\end{align*}
\item Let $BT^*, BA^*, T^*, U^*, A^*$ be the generating functions of the weak basic tautologies of $\mathcal{B}$, weak basic antilogies of $\mathcal{B}$, weak $\mathcal{B}$-tautologies, weak $\mathcal{B}$-unknowns, and weak $\mathcal{B}$-antilogies, respectively. Then the following system of equations is satisfied.
\begin{align*}
T^*(z) &= BT^*(z) + zA^*(z) + z[T^*(z)W(z)+A^*(z)W(z)-A^*(z)T^*(z)],\\
U^*(z) &= mz - BT^*(z) - BA^*(z) + zU^*(z) + zU^*(z)W(z),\\
A^*(z) &= BA^*(z) + zT^*(z) + zA^*(z)T^*(z).
\end{align*}
\end{enumerate}
\end{prop}

Note that these systems of equations have fixed number of equations whenever $m$, the number of variables of the propositional logic system, changes, so they make easy to analyze an asymptotic lower bound as $m\to\infty$. Here, for fixed $\mathcal{B}$, we have
\begin{displaymath}
\lim_{n\to\infty}\frac{[z^n]W_\emptyset(z)}{[z^n]W(z)}\geq \lim_{n\to\infty}\frac{[z^n]T^*(z)}{[z^n]W(z)}\geq \lim_{n\to\infty}\frac{[z^n]T_*(z)}{[z^n]W(z)}
\end{displaymath}

\begin{prop}
\begin{enumerate}[label=(\alph*)]
\item $\mathcal{S}_c$ is the strong basis of $\mathcal{S}_1$.
\item The weak basis of $\mathcal{S}_1$ is the set of well-formed formulae of the form $\psi_1\to[\cdots\to[\psi_k\to p]\cdots ]$ where $\psi_1=p$, and $\psi_2,\cdots,\psi_k$ are not $p$ nor $\mathcal{S}_1$-antilogy. Its generating function satisfies
\begin{displaymath}
B^*(z) = \frac{mz^3}{1+z^2-zW(z)+zA^*(z)}.
\end{displaymath}
which naturally satisfies
\begin{displaymath}
B^*(z) = mz^3-z^2B^*(z)+z[B^*(z)W(z)-B^*(z)A^*(z)].
\end{displaymath}
\end{enumerate}
\end{prop}

Now, we may solve the equation for $\mathcal{S}_1$-strong case algebraically. Since $BA^*(z)=0$, with the identity $A_*(z)=\frac{zT_*(z)}{1-zT_*(z)}$, we obtain
\begin{align*}
T_*(z) &= \frac{1-z^2+zS_c(z)-zW(z)-\sqrt{(1-z^2+zS_c(z)-zW(z))^2-4zS_c(z)(1-zW(z))}}{2z(1-zW(z))}\\
A_*(z) &= \frac{zT_*(z)}{1-zT_*(z)},\\
U_*(z) &= \frac{mz-S_c(z)+zA_*(z)^2}{1-z-z(W(z)+A_*(z))}=\frac{mz-S_c(z)+zA_*(z)(W(z)-T_*(z))}{1-z-zW(z)}.
\end{align*}

For $\rho_0=\frac{1}{2\sqrt{m}+1}$, we have
\begin{align*}
T_*&(\rho_0)=\frac{\sqrt{m}(2m+4\sqrt{m}+3)}{2m+3\sqrt{m}+2}\\
&-\frac{(2\sqrt{m}+1)^2}{\sqrt{m}+1}\sqrt{\frac{m(4m^3+24m^2\sqrt{m}+60m^2+84m\sqrt{m}+70m+33\sqrt{m}+7)}{(2\sqrt{m}+1)^4(2m+3\sqrt{m}+2)^2}}.
\end{align*}
and it is also possible to compute $A_*(\rho_0)$ and $U_*(\rho_0)$. Note that if we substitute $1/y$ for $\sqrt{m}$, then $yT_*(\rho_0)$, $yU_*(\rho_0)$ and $yA_*(\rho_0)$ are analytic about $y$ near 0. So we have series expansions
\begin{align*}
T_*(\rho_0) &= \frac{1}{2\sqrt{m}}-\frac{5}{4m}+\frac{17}{8m\sqrt{m}} + O(\frac{1}{m^2}),\\
A_*(\rho_0) &= \frac{1}{4m}-\frac{3}{4m\sqrt{m}} + O(\frac{1}{m^2}),\\
U_*(\rho_0) &= \sqrt{m}-\frac{1}{2\sqrt{m}}+\frac{1}{m}-\frac{11}{8m\sqrt{m}} + O(\frac{1}{m^2}).
\end{align*}
Then, by \thmr{linearEquation}, if we let $\gamma=\lim_{n\to\infty}\frac{[z^n]S_c(z)}{[z^n]W(z)}$, we have
\begin{align*}
\lim_{n\to\infty}\frac{[z^n]T_*(z)}{[z^n]W(z)}&=\frac{(T_*(\rho_0)-1/\rho_0)(T_*(\rho_0)+\gamma/\rho_0)}{T_*(\rho_0)(1/\rho_0-\sqrt{m}) + A_*(\rho_0) + \sqrt{m}/\rho_0 - 1/\rho_0^2+1}\\
&=\frac{(T_*(\rho_0)-1/\rho_0)(T_*(\rho_0)+\gamma/\rho_0)}{T_*(\rho_0)(\sqrt{m}+1) + A_*(\rho_0) -\sqrt{m}(2\sqrt{m}+3)}\\
&=\frac{1}{m}-\frac{7}{2m\sqrt{m}}+\frac{31}{4m^2}+O(\frac{1}{m^2\sqrt{m}}),
\end{align*}
which gives a slight improvement from $\lim_{n\to\infty}\frac{[z^n]S_1(z)}{[z^n]W(z)}$.

To use this method of undetermined coefficients of power series for weak categories, we need to prove that $yT^*(\rho_0)$, $yU^*(\rho_0)$, $yA^*(\rho_0)$, $yBT^*(\rho_0)$ and $yBA^*(\rho_0)$ are also analytic about $y=\frac{1}{\sqrt{m}}$ near 0. We will prove that our equations have analytic solutions near $y=0$, and there are unique solutions for $BT,BA,T,U$ in a bounded region for fixed small $y$, so our analytic solutions match with real solutions that we want.

First, we will consider the case with arbitrary $BT^*(z)=B^*(z)$ and $BA^*(z)=0$. The equation $U^*(z) = mz - B^*(z) + zU^*(z) + zU^*(z)W(z)$ is actually equivalent to
\begin{displaymath}
mz-B^*(z)=U^*(z)(1-z-zW(z))=\frac{mzU^*(z)}{W(z)}=mz\left(1 - \frac{T^*(z)+A^*(z)}{W(z)}\right).
\end{displaymath}
Moreover, it is easy to check that a system of equations
\begin{align*}
T^*(z) &= B^*(z) + zA^*(z) + z[T^*(z)W(z)+A^*(z)W(z)-A^*(z)T^*(z)],\\
A^*(z) &= zT^*(z) + zA^*(z)T^*(z),
\end{align*}
is actually equivalent to 
\begin{align*}
W(z)B^*(z) &=mz(T^*(z)+A^*(z)), \\
A^*(z) &= zT^*(z) + zA^*(z)T^*(z).
\end{align*}
Then, with
\begin{align*}
\rho_0 &= \frac{1}{2\sqrt{m}+1}=\frac{y}{2+y}= \frac{y}{2}-\frac{y^2}{4}+\frac{y^3}{8}-\frac{y^4}{16}+\cdots,\\
m &= \frac{1}{y^2},\\
W(\rho_0) &= \sqrt{m}=\frac{1}{y},
\end{align*}
we have the system of equations
\begin{align*}
T^*(\rho_0) &= B^*(\rho_0) + \frac{y}{y+2}A^*(\rho_0) + \frac{y}{y+2}\left[\frac{T^*(\rho_0)}{y}+\frac{A^*(\rho_0)}{y}-A^*(\rho_0)T^*(\rho_0)\right],\\
A^*(\rho_0) &= \frac{y}{y+2}T^*(\rho_0) + \frac{y}{y+2}A^*(\rho_0)T^*(\rho_0),
\end{align*}
which is equivalent to 
\begin{align*}
(y+2)B^*(\rho_0) &= T^*(\rho_0)+A^*(\rho_0),\\
A^*(\rho_0) &= \frac{y}{y+2}T^*(\rho_0) + \frac{y}{y+2}A^*(\rho_0)T^*(\rho_0).
\end{align*}
Note that since $B^*,T^*,A^*$ are generating functions, which are bounded by $W$, the values of $T^*(\rho_0)$, $B^*(\rho_0)$, $A^*(\rho_0)$ satisfy $yB^*(\rho_0)$, $yT^*(\rho_0)$, $yA^*(\rho_0)\leq 1$ for each $y=m^{-1/2}$ where $m$ is a positive integer.  Then, we need to solve
\begin{equation}
\begin{split}\label{eqntosolve}
(y+2)[yB^*(\rho_0)] &= [yT^*(\rho_0)]+[yA^*(\rho_0)],\\
[yA^*(\rho_0)] &= \frac{y+[yA^*(\rho_0)]}{y+2}[yT^*(\rho_0)].
\end{split}
\end{equation}
in $[0,1]^3$. Now, assume that we have an equation $B^*(z)=\Theta(B^*(z), T^*(z), A^*(z);m,z,W(z))$, and define $\theta(b,t,a;w)=w\Theta(b/w, t/w, a/w;\frac{1}{w^2},\frac{w}{w+2},\frac{1}{w})$. Then, we define 
\begin{align*}
\lambda(b,t,a;w)&=\left(\theta(b,t,a;w), (w+2)b-a, \frac{w+a}{w+2}t\right),\\
\tilde{\lambda}(b,t,a;w)&=\left(\theta(b,t,a;w), \frac{b}{2}+\frac{(w+1)a+(w+3)t-at}{2(w+2)}, \frac{w+a}{w+2}t\right).
\end{align*}
As we said, the set of fixed points of $\lambda$ and $\tilde{\lambda}$ are same. Now, solving our original system of equations \eqref{eqntosolve} for $yB^*(\rho_0)$, $yT^*(\rho_0)$, $yA^*(\rho_0)$ is equivalent to finding a fixed point of $\lambda$ when $w$ is fixed as $y$. Assume that we have a unique solution $b_0,t_0,a_0$ in $\{(b,t,a)\in\mathbb{C}^3\mid |b|,|t|,|a|\leq 1\}$ satisfying $(b_0,t_0,a_0)=\lambda(b_0,t_0,a_0;0)$, in other words, a fixed point at $w=0$. Since we have $a_0=\frac{a_0t_0}{2}$ and $t_0=2b_0-a_0$, this gives $t_0=2b_0$ and $a_0=0$. Then, for $\epsilon>0$, we say $D\subseteq\mathbb{C}^3$ is a \textbf{proper $\epsilon$-region} if it satisfies following:
\begin{itemize}
\item $D$ is closed and bounded, i.e., compact.
\item $D$ contains an open neighborhood of $(b_0,t_0,a_0)$,
\item $\theta$ is analytic about $w,b,t,a$ when $|w|<\epsilon$ and $(b,t,a)\in D$,
\item $\tilde{\lambda}(D;w)\subseteq D$ when $|w|<\epsilon$,
\item if $y<\epsilon$, then every solution $(yB^*(\rho_0),yT^*(\rho_0), yA^*(\rho_0))$ in $[0,1]^3$ of \eqref{eqntosolve} is in $D$.
\end{itemize}
For the last condition, it is sufficient to show that if $(b,t,a)=\lambda(b,t,a;y)$ and $|b|,|t|,|a|\leq 1$, then $(b,t,a)\in\ D$. Hence, by the analytic implicit function theorem, we will get the existence of analytic solution when the determinant of the Jacobian
\begin{displaymath}
\det J_1=\det\frac{\partial (\text{id} - \lambda)}{\partial(b,t,a)}
\end{displaymath}
is nonzero at $(b_0,t_0,a_0)$ where $w=0$, and by the Banach contraction principle, we will get the uniqueness of the solution for fixed $w=y=m^{-1/2}$ when the Jacobian
\begin{displaymath}
J_2=\frac{\partial\tilde{\lambda}}{\partial(b,t,a)}
\end{displaymath}
has norm value less than 1 whenever $|w|<\epsilon$ and $(b,t,a)\in D$ for some fixed norm. Here, we are using $\tilde{\lambda}$ since the Jacobian of $\lambda$ contains $w+2$ entry, which makes hard to get small norm. By simple computation, we have
\begin{displaymath}
\det J_1(b,t,a;w)=\frac{2+2w+a-t}{2+w}-\frac{2+2w+a-t}{2+w}\frac{\partial\theta}{\partial b}-(2+w-t)\frac{\partial\theta}{\partial t}-(a+w)\frac{\partial\theta}{\partial a}.
\end{displaymath}
and
\begin{displaymath}
J_2(b,t,a;w)=\begin{bmatrix}
\frac{\partial\theta}{\partial b} & \frac{\partial\theta}{\partial t} & \frac{\partial\theta}{\partial a}\\[1.3ex]
\frac{1}{2} & \frac{w+3-a}{2(w+2)} & \frac{w+1-t}{2(w+2)} \\[1.3ex]
0 & \frac{w+a}{w+2} & \frac{t}{w+2}
\end{bmatrix}.
\end{displaymath}
Moreover, if $\Theta$ is a function of $A^*$ only, then we may reduce the number of variables by considering
\begin{displaymath}
\widehat{\lambda}(a;w)=\frac{w+a}{w+2}((w+2)\theta(a;w)-a)=(w+a)\theta(a;w)-\frac{a(w+a)}{w+2},
\end{displaymath}
which gives
\begin{align*}
\widehat{J}_1(a;w) &= \frac{2w+2a+2}{w+2} - a \theta'(a;y) - \theta(a;y) = 1 - \widehat{J}_2(a;w),\\
\widehat{J}_2(a;w) &= a\theta'(a;w) + \theta(a;w) - \frac{w+2a}{w+2}.
\end{align*}
We have free to choose $J_1$ or $\widehat{J}_1$ to check the existence of analytic solution, and $J_2$ or $\widehat{J}_2$ to check the uniqueness of solution. Of course, we need to make a variation for the definition of proper region and choose properly to use $\widehat{J}_2$. Lastly, for the proper $\epsilon$-region with $\epsilon<1$, suppose $(b,t,a)$ is a solution of $(b,t,a)=\lambda(b,t,a;w)$ satisfying $|b|,|t|,|a|\leq 1$ where $|w|<\epsilon$. A proper $\epsilon$-region must contain every such $(b,t,a)$, and we want to find $\epsilon$-region as narrow as possible to get uniqueness easily. Note that we have $a=\frac{w+a}{w+2}t$, which gives
\begin{displaymath}
|a| =\left|\frac{wt}{2+w-t}\right|\leq\frac{|w||t|}{|2+w|-|t|}\leq\frac{\epsilon|t|}{2-\epsilon-|t|}\leq\frac{\epsilon|t|}{1-\epsilon},
\end{displaymath}
so it is reasonable to try to take proper $\epsilon$-region as a subset of $\{(b,t,a)\mid |b|,|t|,|a|\leq 1, |a|\leq\frac{\epsilon}{1-\epsilon}|t|\}$.

Now, consider $\mathcal{S}_1$-weak case. We have two choices of $\Theta(B^*(z),T^*(z),A^*(z);z,m,W(z))$. One is
\begin{displaymath}
\Theta(B^*(z),T^*(z),A^*(z);z,m,W(z)) = mz^3-z^2B^*(z)+z[B^*(z)W(z)-B^*(z)A^*(z)],
\end{displaymath}
and the other is
\begin{displaymath}
\Theta(B^*(z),T^*(z),A^*(z);z,m,W(z))=\frac{mz^3}{1+z^2-zW(z)+zA^*(z)}.
\end{displaymath}
Note that the latter is a function of $A^*$ only. If we take the latter as our $\Theta$, then we have
\begin{displaymath}
\theta(b,t,a;w)=\frac{w^2}{w+2}\cdot\frac{1}{2w^2+3w+2+(w+2)a}.
\end{displaymath}
Since $\theta(b,t,0;0)=0$ always, so $b_0=t_0=a_0=0$ is a unique solution. Now, if $\epsilon\leq\frac{1}{8}$, $|w|<\epsilon$ and $|a|\leq\frac{1}{7}$, then we have
\begin{align*}
|\theta(b,t,a;w)|&\leq\frac{\epsilon^2}{2-|w|}\cdot\frac{1}{2-3|w|-2|w|^2-|2+w||a|}\\
&\leq\frac{\frac{1}{64}}{2-\frac{1}{8}}\cdot\frac{1}{2-\frac{3}{8}-\frac{2}{64}-(2+\frac{1}{8})\frac{1}{7}}=\frac{28}{4335}.
\end{align*}
Hence, if we define
\begin{displaymath}
D=\{(b,t,a)\mid |b|\leq\frac{28}{4335}, |t|\leq \frac{33}{35}, |a|\leq\frac{1}{7}\},
\end{displaymath}
then $D$ is closed, bounded region containing an open neighborhood of $(b_0,t_0,a_0)=(0,0,0)$. Moreover, if $(b,t,a)\in D$, then
\begin{align*}
|\theta(b,t,a;w)|&\leq\frac{28}{4335},\\
\left|\frac{b}{2}+\frac{(w+1)a+(w+3)t-at}{2(w+2)}\right|&\leq\frac{14}{4335}+\frac{(1+\frac{1}{8})\frac{1}{7}+(3+\frac{1}{8})+\frac{1}{7}}{2(2-\frac{1}{8})}=\frac{14}{4335}+\frac{32}{35}\leq\frac{33}{35}<1\\
\left|\frac{w+a}{w+2}t\right|&\leq\frac{\frac{1}{8}+\frac{1}{7}}{2-\frac{1}{8}}=\frac{1}{7},
\end{align*}
and so, $\tilde{\lambda}(D;w)\subseteq D$. Now, if $(b,t,a)=\lambda(b,t,a;w)$ and $|b|,|t|,|a|\leq 1$, then we have
\begin{align*}
|a|&\leq\frac{\epsilon}{1-\epsilon}|t|\leq\frac{1}{7},\\
|b|&=|\theta(b,t,a;w)|\leq\frac{28}{4335},\\
|t|&=\left|\frac{b}{2}+\frac{(w+1)a+(w+3)t-at}{2(w+2)}\right|\leq\frac{33}{35},
\end{align*}
and so, $(b,t,a)\in D$. Thus, $D$ is a proper $\epsilon$-region. Note that we may choose smaller $D$. For example, from $|t|\leq\frac{33}{35}$, we may get $|a|\leq\frac{33}{7\cdot 35}$ and from this bound of $a$, we can get smaller bounds for $t$ and $\theta(b,t,a;w)$. Hence, we may repeat this bootstrap process to make $D$ smaller and smaller. Lastly, we will consider Jacobians. We will choose $J_1,J_2$ rather than $\widehat{J}_1$ and $\widehat{J}_2$. By direct computation, we have
\begin{align*}
\det J_1(b,t,a;w)&=\frac{2+2w+a-t}{2+w}+(w+a)\frac{w^2}{w+2}\cdot\frac{w+2}{(2w^2+3w+2+(w+2)a)^2},\\
J_2(b,t,a;w)&=\begin{bmatrix}
0 & 0 & -\frac{w^2}{(2w^2+3w+2+(w+2)a)^2}\\[1.3ex]
\frac{1}{2} & \frac{w+3-a}{2(w+2)} & \frac{w+1-t}{2(w+2)} \\[1.3ex]
0 & \frac{w+a}{w+2} & \frac{t}{w+2}
\end{bmatrix}.
\end{align*}
First, $\det J_1(0,0,0;0)=1\not=0$, so we have local analytic solution about $w$ from the analytic implicit function theorem. Then, for the (1,3)-entry of $J_2$, we have
\begin{displaymath}
|(J_2)_{13}|\leq\frac{\epsilon^2}{(2-3|w|-2|w|^2-|2+\epsilon||a|)^2}\leq\frac{\frac{1}{8^2}}{\left(2-\frac{3}{8}-\frac{2}{64}-\left(2+\frac{1}{8}\right)\frac{1}{7}\right)^2}=\left(\frac{28}{289}\right)^2,
\end{displaymath}
when $|w|<\epsilon\leq\frac{1}{8}$ and $(b,t,a)\in D$. Now, the sum of the absolute values of the second column is bounded by
\begin{displaymath}
\frac{\epsilon+3+|a|}{2(2-\epsilon)}+\frac{\epsilon+|a|}{2-\epsilon}=\frac{3(\epsilon+|a|+1)}{2(2-\epsilon)}
\end{displaymath}
and of the third column is bounded by
\begin{displaymath}
\left(\frac{28}{289}\right)^2+\frac{\epsilon+1+|t|}{2(2-\epsilon)}+\frac{|t|}{2-\epsilon}=\left(\frac{28}{289}\right)^2+\frac{\epsilon+3|t|+1}{2(2-\epsilon)}.
\end{displaymath}
Here, both of them become less than 1 as $\epsilon\to 0$, so there is $\epsilon_0\leq\frac{1}{8}$ such that $w<\epsilon_0$ implies $\|J_2\|_\infty<1$. Hence, by the Banach contraction principle, we have the uniqueness of the solution for each such $w$, so values of the local analytic solution must match to true values of $yB^*(\rho_0)$, $yT^*(\rho_0)$ and $yA^*(\rho_0)$. Then, $yW(\rho_0)=1$ and $yU^*(\rho_0)=yW(\rho_0)-yT^*(\rho_0)-yA^*(\rho_0)$, so it is also true for $yU^*(\rho_0)$.

From this result, we may assume
\begin{align*}
B^*(\rho_0) &= \frac{b_{-1}}{y} + b_0 + b_1 y + b_2 y^2 + \cdots, \\
T^*(\rho_0) &= \frac{t_{-1}}{y} + t_0 + t_1 y + t_2 y^2 + \cdots, \\
U^*(\rho_0) &= \frac{u_{-1}}{y} + u_0 + u_1 y + u_2 y^2 + \cdots, \\
A^*(\rho_0) &= \frac{a_{-1}}{y} + a_0 + a_1 y + a_2 y^2 + \cdots, 
\end{align*}
where $b_{-1},t_{-1},u_{-1},a_{-1}\geq 0$, since we are considering generating functions. Then, we have a system of quadratic equations
\begin{align*}
B^*(z) &= mz^3-z^2B^*(z)+z[B^*(z)W(z)-B^*(z)A^*(z)],\\
T^*(z) &= B^*(z) + zA^*(z) + z[T^*(z)W(z)+A^*(z)W(z)-A^*(z)T^*(z)],\\
U^*(z) &= mz - B^*(z) + zU^*(z) + zU^*(z)W(z),\\
A^*(z) &= zT^*(z) + zA^*(z)T^*(z),
\end{align*}
and if we write this equation in terms of $y$, we will get
\begin{align*}
B^*(\rho_0) &= \frac{y}{(2+y)^3}-\frac{y^2}{(2+y)^2}B^*(\rho_0)+\frac{1}{2+y}B^*(\rho_0)-\frac{y}{2+y}B^*(\rho_0)A^*(\rho_0),\\
T^*(\rho_0) &= B^*(\rho_0) + \frac{y}{2+y}A^*(x) + \frac{1}{2+y}T^*(\rho_0)+\frac{1}{2+y}A^*(\rho_0)-\frac{y}{2+y}A^*(\rho_0)T^*(\rho_0),\\
U^*(\rho_0) &= \frac{1}{y(2+y)} - B^*(\rho_0) + \frac{y}{2+y}U^*(\rho_0) + \frac{1}{2+y}U^*(\rho_0),\\
A^*(\rho_0) &= \frac{y}{2+y}T^*(\rho_0) + \frac{y}{2+y}A^*(\rho_0)T^*(\rho_0).
\end{align*}
Then, the method of undetermined coefficients gives
\begin{align*}
B^*(\rho_0) &= \frac{1}{4\sqrt{m}}-\frac{1}{2m}+\frac{9}{16m\sqrt{m}}+O(\frac{1}{m^2}), \\
T^*(\rho_0) &= \frac{1}{2\sqrt{m}}-\frac{1}{m}+\frac{5}{4m\sqrt{m}}+O(\frac{1}{m^2}), \\
U^*(\rho_0) &= \sqrt{m}-\frac{1}{2\sqrt{m}}+\frac{3}{4m}-\frac{5}{8m\sqrt{m}}+O(\frac{1}{m^2}), \\
A^*(\rho_0) &= \frac{1}{4m}-\frac{5}{8m\sqrt{m}}+O(\frac{1}{m^2}).
\end{align*}
Now, by \thmr{linearEquation} again, we have
\begin{align*}
\lim_{n\to\infty}\frac{[z^n]B^*(z)}{[z^n]W(z)}&=\frac{1}{4m}-\frac{3}{4m\sqrt{m}}+\frac{9}{8m^2} + O(\frac{1}{m^2\sqrt{m}}),\\
\lim_{n\to\infty}\frac{[z^n]T^*(z)}{[z^n]W(z)}&=\frac{1}{m}-\frac{5}{2m\sqrt{m}}+\frac{29}{8m^2} + O(\frac{1}{m^2\sqrt{m}}).
\end{align*}

This is a the asymptotic density of weak tautologies from simple tautologies, so is a lower bound of the density of tautologies. Also, we may define second kind of simple tautologies as tautologies of $\psi_1,\cdots,\psi_{k+1}\mapsto\psi_{k+2}$ form where
\begin{itemize}
\item $\psi_{k+2}$ is not $\eta_1\to\eta_2$ form,
\item there exists distinct $i,j\leq k+1$ such that $\psi_i=p$ and $\psi_j=\neg p$.
\end{itemize}
Even we introduce this new class of tautologies, from \propr{typeValue} and \propr{wffnorm}, since all simple tautologies of the second kind have $\neg$ symbol in it, we expect that this does not change the $\frac{1}{\sqrt{m}}$ order term of $T^*(\rho_0)$, but it will give an improvement on the $\frac{1}{m}$ order term. Hence, it will not change the $\frac{1}{m}$ order term of ratio, but it will give an improvement on the $\frac{1}{m\sqrt{m}}$ order term of it. Let $\mathcal{S}_2$ be the set of second kind simple tautologies.

We have to start from finding the basis of $\mathcal{S}_1\cup\mathcal{S}_2$. Let us consider weak sense categories, and use simple notations $B,T,U,A$ for generating functions of basis, tautologies, unknowns, and antilogies, respectively. A well-formed formula $\psi_1,\cdots,\psi_{k-1}\mapsto\psi_k$ such that $\psi_k$ is a variable or $\neg\eta$ for a well-formed formula $\eta$ is $(\mathcal{S}_1\cup\mathcal{S}_2)$-basic if and only if one of the following is true.
\begin{itemize}
\item First, $k\geq 2$, and there is a variable $p$ such that $\psi_1,\psi_k$ are $p$, $\psi_2,\cdots,\psi_{k-1}$ are not $p$, and $\psi_2,\cdots,\psi_{k-1}$ are not $(\mathcal{S}_1\cup\mathcal{S}_2)$-antilogies.
\item There is a variable $p$ and $i<k$ such that $\psi_1$ is $p$, $\psi_i$ is $\neg p$, $\psi_k$ is not $p$, $\psi_k$ is not $\neg\eta$ for an $(\mathcal{S}_1\cup\mathcal{S}_2)$-antilogy $\eta$, and for any $1<j<k$, $\psi_j$ is not an $(\mathcal{S}_1\cup\mathcal{S}_2)$-antilogy nor $p$.
\item There is a variable $p$ and $i<k$ such that $\psi_1$ is $\neg p$, $\psi_i$ is $p$, $\psi_k$ is not $p$, $\psi_k$ is not $\neg\eta$ for an $(\mathcal{S}_1\cup\mathcal{S}_2)$-antilogy $\eta$, and for any $1<j<k$, $\psi_j$ is not an $(\mathcal{S}_1\cup\mathcal{S}_2)$-antilogy nor $\neg p$.
\end{itemize}
Also, these three conditions are pairwise disjoint. The generating function for the first case is
\begin{displaymath}
\frac{mz^3}{1-z[W(z)-z-A(z)]},
\end{displaymath}
for the second case is
\begin{displaymath}
mz^2\left(\frac{(m-1)z+z[W(z)-A(z)]}{1-z[W(z)-z-A(z)]} - \frac{(m-1)z+z[W(z)-A(z)]}{1-z[W(z)-z-z^2-A(z)]}\right),
\end{displaymath}
and for the third case is
\begin{displaymath}
mz^3\left(\frac{(m-1)z+z[W(z)-A(z)]}{1-z[W(z)-z^2-A(z)]} - \frac{(m-1)z+z[W(z)-A(z)]}{1-z[W(z)-z-z^2-A(z)]}\right).
\end{displaymath}
Deducing these formulae is similar to the proof of \propr{S1gen}.(a). To apply the method to compute the density of weak tautologies from the $\mathcal{S}_1$ case, we have to consider the existence of proper region $D$. If $\Theta$ is a function of only $A^*(z)$ and $\theta(b,t,0;0)=0$, then to prove the existence of proper region $D$, it is enough to choose $\epsilon>0$ such that there exists $\delta>0$ satisfies
\begin{itemize}
\item if $|w|<\epsilon$ and $|a|\leq\frac{\epsilon}{1-\epsilon}$, then $|\theta(b,t,a;w)|\leq\delta$, and
\item $\frac{\delta}{2} + \frac{3}{2(2-\epsilon)(1-\epsilon)}\leq 1$.
\end{itemize}
If these conditions are satisfied, then $D=\{(b,t,a)\mid |b|\leq\delta, |t|\leq 1,\ |a|\leq\frac{\epsilon}{1-\epsilon}\}$ will be a proper $\epsilon$-region. Then, we may compute Jacobians and check whether $\det J_1(0,0,0;0)$ is nonzero and a norm of $J_2$ is less than 1, where we may reduce $D$ by bootstrap argument and $\epsilon$ freely, if is needed. By direct computation, we can show $\theta(b,t,0;0)=0$ is really true for this case either, and hence, other process to prove analyticity is almost automatic.

After we get the analyticity, we have to consider a system of quadratic equations including the generating function of the basis. We have the following system of equations.
\begin{align*}
B_1(z) =& mz^3 - z^2B_1(z) + z(W(z)B_1(z) - A(z)B_1(z)),\\
B_2(z) =& m(m-1)z^3 + mz^3[W(z)-A(z)] - z^2B_2(z) \\&+ z[W(z)B_2(z) - A(z)B_2(z)],\\
B_3(z) =& m(m-1)z^3 + mz^3[W(z)-A(z)] - (z^2+z^3)B_3(z) \\&+ z[W(z)B_3(z) - A(z)B_3(z)],\\
B_4(z) =& m(m-1)z^4 + mz^4[W(z)-A(z)] - z^3B_4(z) \\&+ z[W(z)B_4(z) - A(z)B_4(z)],\\
B_5(z) =& m(m-1)z^4 + mz^4[W(z)-A(z)] - (z^2+z^3)B_5(z) \\&+ z[W(z)B_5(z) - A(z)B_5(z)],\\
B(z) =& B_1(z) + B_2(z) - B_3(z) + B_4(z) - B_5(z),\\
T(z) =& B(z) + zA(z) + z(T(z)W(z) + A(z)W(z) - A(z)T(z)),\\
U(z) =& mz - B(z) + zU(z) + zU(z)W(z),\\
A(z) =& zT(z) + zA(z)T(z).
\end{align*}
From this system of equations, we have series solutions
\begin{align*}
B_1(\rho_0) =& \frac{1}{4\sqrt{m}}-\frac{1}{2m}+\frac{9}{16m\sqrt{m}}+O(\frac{1}{m^2}),\\
B_2(\rho_0) =& \frac{\sqrt{m}}{4}-\frac{1}{4}-\frac{3}{16\sqrt{m}}+\frac{5}{8m}-\frac{47}{64m\sqrt{m}}+O(\frac{1}{m^2}),\\
B_3(\rho_0) =& \frac{\sqrt{m}}{4}-\frac{1}{4}-\frac{3}{16\sqrt{m}}+\frac{9}{16m}-\frac{35}{64m\sqrt{m}}+O(\frac{1}{m^2}),\\
B_4(\rho_0) =& \frac{1}{8}-\frac{3}{16\sqrt{m}}+\frac{1}{16m}+\frac{3}{32m\sqrt{m}}+O(\frac{1}{m^2}),\\
B_5(\rho_0) =& \frac{1}{8}-\frac{3}{16\sqrt{m}}+\frac{9}{32m\sqrt{m}}+O(\frac{1}{m^2}),\\
B(\rho_0) =& \frac{1}{4\sqrt{m}}-\frac{3}{8m}+\frac{3}{16m\sqrt{m}}+O(\frac{1}{m^2}),\\
T(\rho_0) =& \frac{1}{2\sqrt{m}}-\frac{3}{4m}+\frac{1}{2m\sqrt{m}}+O(\frac{1}{m^2}),\\
U(\rho_0) =& \sqrt{m}-\frac{1}{2\sqrt{m}}+\frac{1}{2m}+O(\frac{1}{m^2}),\\
A(\rho_0) =& \frac{1}{4m}-\frac{1}{2m\sqrt{m}}+O(\frac{1}{m^2}),\\
\end{align*}
and by \thmr{linearEquation}, we get
\begin{align*}
\lim_{n\to\infty}\frac{[z^n]B(z)}{[z^n]W(z)}&=\frac{1}{4m}-\frac{1}{2m\sqrt{m}}+\frac{5}{16m^2} + O(\frac{1}{m^2\sqrt{m}}),\\
\lim_{n\to\infty}\frac{[z^n]T(z)}{[z^n]W(z)}&=\frac{1}{m}-\frac{7}{4m\sqrt{m}}+\frac{5}{4m^2} + O(\frac{1}{m^2\sqrt{m}}),\\
\lim_{n\to\infty}\frac{[z^n]U(z)}{[z^n]W(z)}&=1-\frac{1}{m}+\frac{5}{4m\sqrt{m}}-\frac{1}{8m^2} + O(\frac{1}{m^2\sqrt{m}}),\\
\lim_{n\to\infty}\frac{[z^n]A(z)}{[z^n]W(z)}&=\frac{1}{2m\sqrt{m}}-\frac{9}{8m^2} + O(\frac{1}{m^2\sqrt{m}}).
\end{align*}
This result shows an improvement for the $\frac{1}{m\sqrt{m}}$ order term but not or the $\frac{1}{m}$ order term, as we expected.

For the upper bound of the density, we have a natural upper bound
\begin{displaymath}
1 - \lim_{n\to\infty}\frac{[z^n]A(z)}{[z^n]W(z)}
\end{displaymath}
which gives 
\begin{displaymath}
\lim_{n\to\infty}\frac{[z^n]W_\emptyset(z)}{[z^n]W(z)} \leq 1 - \frac{1}{2m\sqrt{m}}+\frac{9}{8m^2} + O(\frac{1}{m^2\sqrt{m}}).
\end{displaymath}
We may improve the upper bound slightly by dividing the class unknowns into unknowns and not tautologies nor antilogies. In such partitioning, $\mathcal{B}$-tautologies and $\mathcal{B}$-antilogies are not changed, and by same argument, we may compute, with proper analyticity assumption, the density of not tautologies nor antilogies has lower bound
\begin{displaymath}
\frac{1}{4m}-\frac{5}{16m\sqrt{m}}+\frac{5}{32m^2}+O(\frac{1}{m^2\sqrt{m}}),
\end{displaymath}
and this gives an upper bound
\begin{displaymath}
1 -\frac{1}{4m} - \frac{3}{16m\sqrt{m}}+\frac{31}{32m^2} + O(\frac{1}{m^2\sqrt{m}}).
\end{displaymath}
But this upper bound is still too far from the lower bound. So we need some different approach.

Here, since the set of tautologies is $\cap_{v\in\mathcal{VA}}T^v$, the minimum density of $T^v$ among every truth assignment $v$ gives an upper bound. By symmetry, it only depends on the number of variable assigned as true. Suppose $v$ has $m-k$ true variables and $k$ false variables. Then, the generating function $F$ for false well-formed formulae satisfies the following equation.
\begin{displaymath}
F(z) = kz + z(W(z)-F(z))+z(W(z)-F(z))F(z).
\end{displaymath}
Hence, we have
\begin{displaymath}
F(z)=\frac{-1-z+zW(z)+\sqrt{(1+z-zW(z))^2+4z^2(k+W(z))}}{2z}.
\end{displaymath}
By the Szeg\H{o} lemma, we have
\begin{displaymath}
\lim_{n\to\infty}\frac{[z^n]F(z)}{[z^n]W(z)}=\frac{1}{2}-\frac{\sqrt{m}}{2\sqrt{m+6\sqrt{m}+4k+4}}
\end{displaymath}
so it gives
\begin{displaymath}
\lim_{n\to\infty}\frac{[z^n]W_\emptyset(z)}{[z^n]W(z)}\leq\frac{1}{2}+\frac{\sqrt{m}}{2\sqrt{m+6\sqrt{m}+4k+4}}.
\end{displaymath}
Thus, if we set $k=m$, then we get
\begin{displaymath}
\lim_{n\to\infty}\frac{[z^n]W_\emptyset(z)}{[z^n]W(z)}\leq\frac{1}{2}+\frac{\sqrt{m}}{2\sqrt{5m+6\sqrt{m}+4}}=\frac{\sqrt{5}+1}{2\sqrt{5}}-\frac{3}{10\sqrt{5m}}+\frac{7}{100\sqrt{5} m} + O(\frac{1}{m\sqrt{m}}).
\end{displaymath}
This upper bound is still not satisfactory, but it successfully decreases the highest term. As a result, the density of tautology has
\begin{displaymath}
\frac{1}{m} + O(\frac{1}{m\sqrt{m}}) \leq \lim_{n\to\infty}\frac{[z^n]W_\emptyset(z)}{[z^n]W(z)} \leq \frac{\sqrt{5}+1}{2\sqrt{5}} + O(\frac{1}{\sqrt{m}})
\end{displaymath}
as asymptotic behaviors, where the lower bound is conjectured as tight. This is the remaining problem about the asymptotic behavior.

Of course, we have remaining problems for other sections, too. For the \secr{exactcomp}, we could compute the density of the tautologies in the propositional logic system with $m$ variables, negation, and implication for the case $m=2,3,4$; and our method can be easily applied to other logical symbols, but for the case $m\leq 4$ or 5. The modern computation power is not enough for large $m\geq 6$. Here, the complexity of the exact algebraic formula is essential, so the limitation to getting the exact algebraic formula is unavoidable. But there is still a chance to find a realistic method to compute approximate values that is close to the real value as we want.

For the \secr{analytic}, we suggested some conditions under which the $s$-cut concept works, but it is not universal and hard to check if they are satisfied. We practically verified the result, but it is not proved satisfactorily. Hence, we will ask to find better conditions that when shifted $s$-cut operators are guaranteed to have the Jacobian less than 1, and rigorous proof that $s$-cut solutions give a better approximation than the raw ratios in most cases. 

This work was partially supported by the National Research Foundation of Korea (NRF) grant funded by the Korea government (MSIT) (No. 2019R1F1A1062462).

\bibliographystyle{plain}
\bibliography{Tautologies}

\end{document}